%% file: paper.tex
\documentclass[final,hidelinks,onefignum,onetabnum]{siamart220329}

\usepackage{pdfpages}

\input{VP_preamble_shared}

\ifpdf
\hypersetup{
  pdftitle={Mean-Field Approximation of Dynamics on Networks},
  pdfauthor={J.A. Ward, G. Timar, P.L. Simon}
}
\fi

\begin{document}
\graphicspath{{./figures/}}

\maketitle

\begin{abstract}
Many real-world phenomena can be modelled as dynamical processes on networks, a prominent example being the spread of infectious diseases such as COVID-19. Mean-field approximations are a widely used tool to analyse such dynamical processes on networks, but these are typically derived using plausible probabilistic reasoning, introducing uncontrolled errors that may lead to invalid mathematical conclusions. In this paper we present a rigorous approach to derive mean-field approximations from the exact description of Markov chain dynamics on networks through a process of averaging called approximate lumping. We consider a general class of Markov chain dynamics on networks in which each vertex can adopt a finite number of ``vertex-states'' (e.g. susceptible, infected, recovered etc.), and transition rates depend on the number of neighbours of each type. Our approximate lumping is based on counting the number of each type of vertex-state in subsets of vertices, and this results in a density dependent population process. In the large graph limit, this reduces to a low dimensional system of ordinary differential equations, special cases of which are well known mean-field approximations. Our approach provides a general framework for the derivation of mean-field approximations of dynamics on networks that unifies previously disconnected approaches and highlights the sources of error.
\end{abstract}



\section{Introduction}
Dynamical processes on networks are important and widely studied \cite{barrat2008dynamical,newman2018networks,porter2016dynamical,kiss2017book}. They have been used to study real-world phenomena, such as epidemics \cite{kiss2017book,pastor2015epidemic}, opinion dynamics \cite{galam2002minority,sood2005voter,sznajd2000opinion}, and spin systems with critical phenomena \cite{dorogovtsev2002ising,leone2002ferromagnetic,dorogovtsev2008critical}. Many such models can be described mathematically as Markov chains \cite{simon2011exact,ward2018general,ward2019exact}, but often their state-space is so large it is impossible to use mathematical tools from the theory of Markov chains. Instead it is standard to make use of ``mean-field'' approximations \cite{vazquez2008analytical,pastor2001epidemic,kiss2017book,gleeson2011high,fennell2017multistate,metz2025dynamical,crisanti2018path,hertz2016path}, in which aspects of network structure and dynamical correlations are ignored \cite{gleeson2012accuracy}.

A class of mean-field approximations, usually referred to as dynamical mean-field, concerns dynamics on networks encoded in the form of sets of Langevin-type or deterministic equations, where the disorder (in either network structure or interaction strengths) is generally averaged out using path-integral methods \cite{metz2025dynamical,crisanti2018path,de2022dynamical}.
Most mean-field approximations that attempt to provide a low-dimensional description of stochastic dynamics starting from the original Markov chain, however, tend to be based on plausible probabilistic reasoning, thus lacking a rigorous mathematical foundation \cite{vazquez2008analytical,pastor2001epidemic,kiss2017book,gleeson2011high,fennell2017multistate}.
Such approximations have the potential to introduce uncontrolled errors that limit the potential for mathematical analysis, since the mean-field approximation is not faithful to the original process. For example, controversy concerning the critical epidemic threshold in scale-free networks stemmed from the use of mean-field approximations \cite{pastor2001epidemic,gomez2008spreading,chatterjee2009contact,boguna2013nature}. Moreover, the assumptions that underpin mean-field approximations---the absence of clustering, modularity/community structure, dynamical correlations---are routinely violated by dynamical processes on real-world complex networks \cite{gleeson2012accuracy}. Thus it is difficult to know when a mean-field approximation will be accurate or how the error depends on the network structure or the dynamic \cite{ward2022micro}. Consequently, the quantification of approximation error has been recognised as a key challenge for network epidemic modellers \cite{pellis2015eight}.

In this paper, we develop a mathematical foundation for mean-field approximation that starts with the exact Markov chain description of a broad class of dynamical processes on networks where interactions are governed by vertices' local neighbourhoods. We use a technique called approximate lumping to derive a ``density dependent population process'' that in the large network limit converges to a relatively small set of differential equations. There are three significant advances on previous work \cite{ward2022micro}. Firstly, by basing the approximate lumping on a partition of vertices, we are able to incorporate network structure into the mean-field approximation in a very flexible way. Secondly, we combine two key techniques---approximate lumping and convergence of density dependent population processes---to connect the exponentially-large but exact micro-scale Markov chain, through to a highly reduced system of ordinary differential equations. Thirdly, we show that both degree-based and individual-based mean-field approximations can be captured in a unified way through our general framework. Crucially, our approach is rigorous, elucidates the averaging process and highlights sources of error.

We start by describing Markov chain dynamics on networks in Section~\ref{sec:background} and approximate lumping in Section~\ref{sec:lumping}. In Section~\ref{sec:vertexpartitionlumping} we describe how we use a partition of vertices to define an approximate lumping. The combinatorics to derive our mean-field approximation is involved, so in Section~\ref{sec:twocell} we describe the process for the case where the vertices are partitioned into two sets, before we generalise this to an arbitrary finite number of partitions in Section~\ref{sec:Ppartitions}. 
Furthermore, we present a simple example in the corresponding sections of the paper's Supplementary Materials. Readers interested in the working details of our approach may find it helpful to read Sections~SM2--SM5 of the Supplementary Materials as they read the corresponding sections of the main paper. We treat some special cases of our density dependent population process in Section~\ref{sec:applications}, then derive the large network limit mean-field equations in Section~\ref{sec:largeN}. In Sections~\ref{sec:dbMF} and \ref{sec:IBMF} we are then able to derive degree-based and individual-based mean-field approximations respectively, and we consider degree-based mean-field approximation of the configuration model in Section~\ref{sec:configmodel}. Finally we discuss our findings in Section~\ref{sec:discussion}.

\section{Mathematical background}
\label{sec:background}
Let $G=(V,E)$ denote a graph or network with vertex set $V$ and edge set $E\subset V\times V$, where the number of vertices is $N=|V|$. Unless otherwise stated, we consider dynamical processes on finite simple networks (i.e. undirected, unweighted with no self-loops or multiple edges) described by continuous-time Markov chains where each vertex can be in one of a finite number $M$ of \emph{vertex-states} and the set of possible vertex-states is $\mathcal{W}=\{\mathcal{W}_1,\mathcal{W}_2,\dots,\mathcal{W}_M\}$.

\subsection{State-space}  
The state-space of the Markov chain is the set of all permutations of $N$ vertex-states chosen from $\mathcal{W}$ with repetition. This is equivalent to $\mathcal{S}=\mathcal{W}^V$, i.e. the set of all functions from $V$ to $\mathcal{W}$, and so if the network is in state $S\in \mathcal{S}$ then the vertex-state of vertex $v\in V$ is $S(v)$. From here onwards we will refer to the states $S\in \mathcal{S}$ as \emph{microstates}, to clearly distinguish them from vertex-states and \emph{lumped states}, which will be introduced in Section~\ref{sec:lumping}. The number of microstates in $\mathcal{S}$ is $M^N$, where $N$ is the number of vertices, which is extremely large for even moderate $N$. However, since $\mathcal{S}$ is finite we can enumerate the microstates so that $\mathcal{S}=\{S^{[1]},S^{[2]},\dots,S^{[M^N]}\}$.

\subsection{Transitions}
We assume that changes in microstate correspond to a single vertex $v\in V$ changing its vertex-state, and the rate that this occurs is a function of only the number of $v$'s neighbours in each of the vertex-states. We also assume that this rate function is the same for all vertices. Thus we assume model dynamics are driven by local interactions captured via a collection of functions between vertex-states.
Note that a more general class of models would allow for behaviours in which multiple vertices change vertex-states at once, for example if a vertex exports its vertex-state to its neighbours \cite{ward2018general}. 
We will now give a precise definition of the network dynamics that we consider.
\begin{definition}
For a finite non-empty set of vertex-states $\mathcal{W}$ and $\mathcal{A},\mathcal{B}\in\mathcal{W}$, let 
$$\mathbf{R}_{\mathcal{A},\mathcal{B}}:\mathbb{Z}^M_{\ge 0}\rightarrow\mathbb{R}_{\ge0}.$$
A vertex-state transition matrix (VSTM) $\mathbf{R}$ is the collection of functions $\mathbf{R}_{\mathcal{A},\mathcal{B}}$ for each $\mathcal{A},\mathcal{B}\in\mathcal{W}$.
\end{definition}
In the models we consider, $\mathbf{R}_{\mathcal{A},\mathcal{B}}(n_1,n_2,\dots,n_{M})\ge0$ gives the rate that a vertex in vertex-state $\mathcal{A}$ changes to vertex-state $\mathcal{B}$ if it has $n_1$ neighbours in vertex-state $\mathcal{W}_1$, $n_2$ neighbours in vertex-state $\mathcal{W}_2$, etc. 
If transitions between a pair of vertex-states $\mathcal{A},\mathcal{B}\in\mathcal{W}$ do not occur in a particular model, then the rate $\mathbf{R}_{\mathcal{A},\mathcal{B}}$ is identically zero. 
\begin{definition}
A homogeneous Single-Vertex Transition model (SVT) is a pair $(\mathcal{W},\mathbf{R})$.
\end{definition}
We think of an SVT $\mathcal{M}=(\mathcal{W},\mathbf{R})$ as a directed graph over vertex-states where a directed edge goes from vertex-state $\mathcal{A}$ to $\mathcal{B}$ if $\mathbf{R}_{\mathcal{A},\mathcal{B}}$ is not identically zero.
In this paper, our main focus will be on models whose VSTMs are affine functions, so 
\begin{equation}
\mathbf{R}_{\mathcal{A},\mathcal{B}}(n_1,n_2,\dots,n_{M})=\zeta_0^{\mathcal{A},\mathcal{B}}+\sum_{m=1}^{M}\zeta_m^{\mathcal{A},\mathcal{B}}n_m,
\label{eq:affinevstm}
\end{equation}
where all of the constants $\zeta_m^{\mathcal{A},\mathcal{B}}$ are non-negative. We will refer to SVTs with affine VSTMs as affine SVTs.
Most SVTs have VSTMs of this form \cite{ward2019exact}, although notable exceptions include the non-zero temperature Ising-Glauber dynamics \cite{glauber1963time}, the nonlinear $q$-voter model \cite{castellano2009nonlinear} and threshold models \cite{watts2002simple}. 

\subsection{Kolmogorov equations: infinitesimal generator}
Given the network $G$ and model $\mathcal{M}$, we need to define the corresponding continuous-time Markov chain. Let $X(t)=(X_1(t),X_2(t),\dots,X_{M^N}(t))^{\rm T}$ be the time-dependent Markov chain probability distribution over $\mathcal{S}$, where $X_i(t)$ is the probability of being in microstate $S^{[i]}$ at time $t$. The evolution of $X(t)$ is then given by the forward Kolmogorov or master equation \cite{kijima1997markov}, $$\dot{X}=\mathbf{Q}^{\rm T}X,$$ where $\mathbf{Q}$ is the infinitesimal generator, an $M^N$ by $M^N$ matrix in which each off-diagonal component $\mathbf{Q}_{kl}$ gives the transition rate from $S^{[k]}$ to $S^{[l]}$, and the diagonal components ensure that rows sum to zero so that probability is conserved. We assume that a vertex changes vertex-state instantaneously, thus transitions only occur between pairs of microstates that differ in exactly one vertex-state. 
\begin{definition}
A pair of states $S^{[k]},S^{[l]}\in\mathcal{S}$ forms a {\bf transition pair} with transition vertex $v$, denoted $S^{[k]}\stackrel{v}{\sim}S^{[l]}$, if $S^{[k]}(v)\ne S^{[l]}(v)$ and $S^{[k]}(u)=S^{[l]}(u)$ for all $u\ne v$.
\end{definition}
For vertex $v$ and microstate $S^{[k]}$ let $n^{[k]}(v)=(n^{[k]}_1(v),n^{[k]}_2(v),\dots,n^{[k]}_{M}(v))$, where $n^{[k]}_m$ is the number of neighbours of $v$ with vertex-state $\mathcal{W}_m\in\mathcal{W}$ in microstate $S^{[k]}$.
Thus for $S^{[k]},S^{[l]}\in\mathcal{S}$ and $S^{[k]}\ne S^{[l]}$, the transition rate from $S^{[k]}$ to $S^{[l]}$ in an SVT is given by
\begin{equation*}
\mathbf{Q}_{kl}=\left\{\begin{array}{cc} \mathbf{R}_{S^{[k]}(v),S^{[l]}(v)}(n^{[k]}(v)) 
& \text{if $S^{[k]}\stackrel{v}{\sim}S^{[l]}$}\\ 0 &
\text{otherwise}
\end{array}\right.,
\end{equation*}
where vertex $v$ is the transition vertex (if the states $S^{[k]}$ and $S^{[l]}$ form a transition pair) that goes from vertex-state $S^{[k]}(v)$ to $S^{[l]}(v)$.
  
\section{Coarse-graining via lumping: theoretical foundation}
\label{sec:lumping}
We consider \emph{lumping} of Markov chains \cite{kemeny1960finite}. 
Let $\Pi_\mathcal{S}=\{\mathcal{S}_1,\mathcal{S}_2,\dots,\mathcal{S}_{n}\}$ be a partition of microstate-space, so $\mathcal{S}_i\cap\mathcal{S}_j=\emptyset$ for each $i\ne j$, and $\cup_i\mathcal{S}_i=\mathcal{S}$. An \emph{exact lumping} is a partition of microstate-space $\Pi_\mathcal{S}$ that preserves the Markov property, a necessary and sufficient condition for which is that the sum of transition rates from microstate $S^{[k]}\in \mathcal{S}_i$ to microstates in the cell $\mathcal{S}_j$, i.e. $$\sum_{S^{[l]}\in S_j}\mathbf{Q}_{kl},$$ is the same for all microstates $S^{[k]}$ in the cell $\mathcal{S}_i$.
In matrix notation \cite{ward2019exact}, this is equivalent to the existence of an $n\times n$ matrix $\mathbf{q}$ such that
\begin{equation}
\mathbf{QC}=\mathbf{Cq},
\label{eq:lumpability}
\end{equation}
where $\mathbf{C}\in\{0,1\}^{M^N\times n}$ is the \emph{collector matrix} \cite{buchholz1994exact} whose $kj$th component is
\begin{equation}
\mathbf{C}_{kj}=\left\{\begin{array}{cc}
    1 & \text{if $S^{[k]}\in \mathcal{S}_j$},\\
    0 & \text{otherwise}. \end{array}\right.
\label{eq:col}
\end{equation}
The collector matrix collects those microstates in a column that belong to the same cell, or in other words the same ``lumped state'', in the partition. 

We call \eqref{eq:lumpability} the \emph{lumpability condition}.
Note that $\mathbf{q}$ can be given explicitly for an exact lumping by introducing the \emph{distributor matrix} \cite{buchholz1994exact} $\mathbf{D}\in\mathbb{R}^{n\times M^N}$, whose $il$th component is
\begin{equation}
\mathbf{D}_{il}=\left\{\begin{array}{cc}
    \frac{1}{|\mathcal{S}_i|} & \text{if $S^{[l]}\in \mathcal{S}_i$},\\
    0 & \text{otherwise}. \end{array}\right.
\label{eq:dis}
\end{equation}
Specifically, $\Pi_{\mathcal{S}}$ satisfies the lumpability condition when $\mathbf{Q}$ commutes with $\mathbf{CD}$ \cite{ward2019exact}. Note that $\mathbf{DC}=\mathbf{I}$, the identity matrix, hence multiplying \eqref{eq:lumpability} by $\mathbf{D}$ we get the generator $\mathbf{q}$ of the lumped system explicitly as
\begin{equation}
\mathbf{q}=\mathbf{DQC} .
\label{eq:qDQC}
\end{equation}

We use $x(t)=(x_1(t),\dots,x_n(t))^T$ to denote the time dependent Markov chain probability distribution over $\Pi_\mathcal{S}$, where $x_i(t)$ is the probability of being in the lumping partition cell $\mathcal{S}_i$.
For this reason, we will refer to $\mathcal{S}_i$ as a \emph{lumped state}.
When the lumpability condition is satisfied, the evolution of $x(t)$ is determined by the lumped master equation
\begin{equation}
\dot{x}=\mathbf{q}^{T}x,
\label{eq:Ydot}
\end{equation}
and if $x(0)=\mathbf{C}^{\rm T}X(0)$ then we have $x(t)=\mathbf{C}^{\rm T}X(t)$ for all $t$. In other words, for each $\mathcal{S}_i\in\Pi_{\mathcal{S}}$, the sum of the probabilities of being in each microstate in $\mathcal{S}_i$ at time $t$ is equal to $x_i(t)$.

A lumping that does not satisfy the lumpability condition is an \emph{approximate lumping} \cite{buchholz1994exact}. Given a partition $\Pi_\mathcal{S}$ of microstate-space that does not satisfy the lumpability condition \eqref{eq:lumpability}, our approach is to still use the set of lumped states $\Pi_{\mathcal{S}}$ and the corresponding generator $\mathbf{q}=\mathbf{DQC}$, then solve the lumped master equation \eqref{eq:Ydot} for $x(t)$. Note that while this defines a Markov chain, it does not directly correspond to the SVT that it has been derived from, so we do not expect $x(t)$ to equal $\mathbf{C}^{\rm T}X(t)$ for all $t$. We will however assume that the initial condition of the approximate lumping can be chosen so that $x(0)=\mathbf{C}^{\rm T}X(0)$.
Using the definition of $\mathbf{C}$ \eqref{eq:col}, we have 
\begin{equation}
(\mathbf{QC})_{kj}=\sum_{l=1}^{M^N}\mathbf{Q}_{kl}\mathbf{C}_{lj} = \sum_{S^{[l]}\in \mathcal{S}_j}\mathbf{Q}_{kl}, \label{eq:rij}
\end{equation}
i.e. $(\mathbf{QC})_{kj}$ is the sum of the rates out of the microstate $S^{[k]}$ into the $j^{\textrm{th}}$ lumped state when $S^{[k]}\not\in \mathcal{S}_j$ and minus the sum of rates out of microstate $S^{[k]}$ when $S^{[k]}\in \mathcal{S}_j$. Then using the definition of the distributor matrix \eqref{eq:dis} we have
\begin{equation}
    \mathbf{q}_{ij}=(\mathbf{DQC})_{ij}= \sum_{k=1}^{M^N}\mathbf{D}_{ik}(\mathbf{QC})_{kj} = \sum_{k=1}^{M^N}\mathbf{D}_{ik}\sum_{S^{[l]}\in \mathcal{S}_j}\mathbf{Q}_{kl} = \frac{1}{|\mathcal{S}_i|}\sum_{S^{[k]}\in \mathcal{S}_i}\sum_{S^{[l]}\in \mathcal{S}_j}\mathbf{Q}_{kl}.%
    \label{eq:WQVij}
\end{equation}
Thus $\mathbf{q}_{ij}$ is the average of the sum of rates out of microstates in the $i$th lumped state and into the $j$th lumped state. 

Summarising, we can say that starting from the full infinitesimal generator, $\mathbf{Q}$, and choosing a partition of the state-space, equation \eqref{eq:WQVij} yields the infinitesimal generator $\mathbf{q}$ of the lumped (coarse-grained) system. Note that the partition of the microstate-space is encoded in the collector and distributor matrices, $\mathbf{C}$ and $\mathbf{D}$ respectively.

\section{Lumping based on vertex set partitions}
\label{sec:vertexpartitionlumping}
In the previous section we introduced the notion of lumping in general, however we have not yet described how we will determine the partition of microstate-space, on which the lumping is based. This will be dealt with in this section. 

\subsection{Motivating example}
As a simple example, consider an SIS epidemic on a graph with three nodes. Before we consider specific graphs, we first define some notation. We will use $\mathcal{B}$ (for blue) to denote susceptible vertices and $\mathcal{R}$ (for red) to denote infected vertices, and we will represent a microstate with three such letters. An important property of a useful lumping is that two microstates in the same partition cell should have the same number of infected nodes. This ensures that the total number of infected nodes (i.e. the prevalence) can be determined from the lumped system as well. The simplest lumping that preserves this property contains the following four lumping classes: $\mathcal{S}_1 = \{ \mathcal{BBB}\}$, $\mathcal{S}_2 = \{ \mathcal{BBR},\mathcal{BRB},\mathcal{RBB} \}$, $\mathcal{S}_3 = \{ \mathcal{BRR},\mathcal{RBR},\mathcal{RRB} \}$, $\mathcal{S}_4 = \{ \mathcal{RRR} \}$. 
Generalizing this idea to arbitrary dynamics, it has been shown that approximate lumping based on partitions of microstate-space into sets of microstates with the same number of vertices in each vertex-state result in mean-field birth-death processes for $M=2$ and mean-field population models for $M>2$ \cite{ward2022micro}. 

Returning to the example above, one can realise that the lumping does not take into account the graph structure at all. For example, having a path graph with three nodes, the degree of the central node is two, while that of the nodes at the left and right end is only one, which could be reflected by the lumping. A natural choice of lumping with a finer partition is the following: 
$\mathcal{S}_1 = \{ \mathcal{BBB} \}$, $\mathcal{S}_2 = \{ \mathcal{BBR},\mathcal{RBB} \}$, $\mathcal{S}_3 = \{ \mathcal{BRB} \}$, $\mathcal{S}_4 = \{ \mathcal{BRR},\mathcal{RRB} \}$, $\mathcal{S}_5 = \{ \mathcal{RBR} \}$, $\mathcal{S}_6 = \{ \mathcal{RRR} \}$. 
In this lumping partition, the central node plays a different role than the end nodes. This will be interpreted as follows: the nodes are divided into two groups, the central node and the end nodes. The lumping is based on the number of infected nodes in these groups. The number of infected nodes in the first group can be 0 or 1 while in the second group it can be 0, 1 or 2. Thus the lumping classes above can be encoded by the number of infected nodes in the two groups as follows: $\mathcal{S}_1$ by $s^{[1]}=(0,0)$, $\mathcal{S}_2$ by $s^{[2]} =(0,1)$, $\mathcal{S}_3$ by $s^{[3]}=(1,0)$, $\mathcal{S}_4$ by $s^{[4]}=(1,1)$, $\mathcal{S}_5$ by $s^{[5]} =(0,2)$ and $\mathcal{S}_6$ by $s^{[6]}=(1,2)$. (Note that our notation distinguishes between subsets of microstates in a lumping class, such as $\mathcal{S}_1$, and the encoding of that lumping class according to the number of infected vertices in each group, such as  $s^{[1]}$.) This encoding can be extended by the number of susceptible nodes, which in this example is redundant information, so that for example
$$
\mathbf{s}^{[5]} =\left(\begin{array}{cc}
1 & 0\\
0 & 2
\end{array}\right),
$$
where the first row of the matrix shows the number of susceptible nodes in the two groups, while the second row contains the number of infected nodes. The other lumped states can be similarly encoded by $2\times 2$ matrices. This notation will be used below for the general case.

\subsection{Microstate-space partition based on a vertex set partition}
Now we generalise the above approach by considering partitions of the \emph{vertex} set, where the lumped states are based on partitions of microstate-space into sets of microstates with the same number of vertices in each vertex-state \emph{within} each of the cells of the \emph{vertex} partition. This defines the lumped states that we consider, and so in the remainder of the paper we will primarily refer to lumped states in terms of these counts, rather than cells in the partition of microstate-space $\Pi_{\mathcal{S}}$ (like $\mathcal{S}_i$, which are subsets of microstates).

Let $\Pi_V=\{V_1,\dots,V_P\}$ be a partition of the vertex set $V$, such that $V_p\cap V_q=\emptyset$ for $p\ne q$ and $\cup_{p}V_p=V$.
We now define the lumped macrostate-space and the corresponding partition of microstate-space precisely.
\begin{definition}
For a network $G=(V,E)$, with $N=|V|$, and an SVT $\mathcal{M}=(\mathcal{W},\mathbf{R})$, with $M=|\mathcal{W}|$, let $\Pi_V$ be a vertex-partition, where $P=|\Pi_V|$ and $N_p=|V_p|$ for each $V_p\in\Pi_V$. 
The corresponding \emph{vertex-partition macrostate-space} is the set of non-negative, integer valued matrices $\mathbf{s}\in\mathbb{Z}^{M\times P}_{\ge0}$ such that
\begin{equation}
\sum_{m=1}^M \mathbf{s}_{m,p}=N_p.
\label{eq:sumvstates}
\end{equation}
\label{def:vpppms}
\end{definition}
Note that since $\sum_{p=1}^PN_p=N$, we have $\sum_{p=1}^P\sum_{m=1}^M\mathbf{s}_{m,p}=N.$
\begin{definition}
Let $(\mathbf{s}^{[1]},\mathbf{s}^{[2]},\dots,\mathbf{s}^{[n]})$ be an enumeration of a vertex-partition macrostate-space. A \emph{vertex-partition lumping} is a partition of microstates $\Pi_{\mathcal{S}}=\{\mathcal{S}_1,\mathcal{S}_2,\dots,\mathcal{S}_n\}$ such that for $\mathcal{S}_i\in\Pi_{\mathcal{S}}$, the number of vertices in vertex-state $\mathcal{W}_m\in\mathcal{W}$ in vertex-partition cell $V_p\in\Pi_V$ is $\mathbf{s}^{[i]}_{m,p}$.
\label{def:vpal}
\end{definition}

\section{The lumped generator for two cell vertex-partitions}
\label{sec:twocell}
We now describe the approximate lumping approach using two vertex-partition cells for finite graphs and homogeneous SVTs. 
The fully general case with a finite number of vertex-partitions cells will be presented in Section~\ref{sec:Ppartitions}. 
Let us consider the case $P=2$ in Definition \ref{def:vpppms}, i.e. denote the vertex partition by $\Pi_{V}=\{V_1,V_2\}$, where $V_1\cap V_2=\emptyset$ and $V_1\cup V_2=V$; also $N_1=|V_1|$ and $N_2=|V_2|$. 
For finite $M$ and $P=2$, a lumped state will be denoted by a matrix $\mathbf{s}\in\mathbb{Z}^{M\times P}_{\ge0}$ whose $m,p$th entry, $\mathbf{s}_{m,p}$, is the number of vertices in vertex-state $\mathcal{W}_m$ in the vertex partition cell $V_p$. Thus summing the entries in the $p$th column of $\mathbf{s}$ yields $N_p$, as in \eqref{eq:sumvstates}, and summing all entries yields $N$.
A lumped state corresponds to choosing $N_1$ vertices from the $M$ possible vertex-states with repetition in the first partition and $N_2$ vertices from the $M$ possible vertex-states with repetition in the second partition. Thus the total number of lumped states is 
\begin{equation}
n={N_1+M-1\choose N_1}{N_2+M-1\choose N_2}\label{eq:nlumpedstates}.
\end{equation}
Consequently we can number the lumped states as  $\mathbf{s}^{[1]},\mathbf{s}^{[2]},\dots,\mathbf{s}^{[n]}$, and the lumping partition is then $\Pi_{\mathcal{S}}=\{\mathcal{S}_1,\mathcal{S}_2,\dots,\mathcal{S}_{n}\}$, where 
$$\mathcal{S}_i=\left\{S^{[j]}\in \mathcal{S}\;\middle|\;\sum_{v\in V_p}
\delta_{S^{[j]}(v),\mathcal{W}_m}=\mathbf{s}_{mp}^{[i]},
\;\mathcal{W}_m\in\mathcal{W}, V_p\in\Pi_V\right\},$$
and $\delta_{\mathcal{A},\mathcal{B}}$ is the Kronecker delta function.
Note that for a given microstate $S^{[j]}$, this simply counts the number of vertices in each vertex-state and in each vertex partition cell, and checks whether these match the corresponding values in $\mathbf{s}^{[i]}$; if they do, then microstate $S^{[j]}$ is assigned to partition cell $\mathcal{S}_i$.
We use $\mathbf{s}_p=(\mathbf{s}_{1,p},\mathbf{s}_{2,p},\dots,\mathbf{s}_{M,p})^{\rm T}$
to denote the $p$th column of $\mathbf{s}$. We will only use a single subscript on $\mathbf{s}$ when referring to this vector.

\subsection{The lumped generator for an arbitrary two cell vertex-partition}
We will start by considering the transition rate from an arbitrary lumped state $\mathbf{s}^{[i]}$ to another arbitrary lumped state $\mathbf{s}^{[j]}\ne\mathbf{s}^{[i]}$. In Section~\ref{sec:lumping}, we saw that in general this is given by 
\begin{equation}
    \mathbf{q}_{ij}=\frac{1}{|\mathcal{S}_i|}\sum_{S^{[k]}\in \mathcal{S}_i}\sum_{S^{[l]}\in \mathcal{S}_j}\mathbf{Q}_{kl},
    \label{eq:qij}
\end{equation}
but it turns out that rather than summing over microstates, as \eqref{eq:qij} suggests, it is easier to consider the possible transitions of individual vertices and sum their rates.

First note that the number of arrangements of vertex-states over vertices in the $p$th cell that correspond to $\mathbf{s}^{[i]}_p$ is the multinomial
$${N_p \choose \mathbf{s}^{[i]}_p}:={N_p \choose \mathbf{s}^{[i]}_{1,p}\; \mathbf{s}^{[i]}_{2,p} \;\cdots\; \mathbf{s}^{[i]}_{M,p}},$$
so the total number of microstates in the lumping partition cell $\mathcal{S}_i$ (which corresponds to the lumped state $\mathbf{s}^{[i]}$) is
$$|\mathcal{S}_i|={N \choose \mathbf{s}^{[i]}}:=\prod_{p=1}^P{N_p \choose \mathbf{s}^{[i]}_p}.$$ 
Note we will always have a scalar on the top part of our generalised multinomial notation. When we have a vector on the bottom part, the notation corresponds to a multinomial over the entries in the vector, and when we have a matrix on the bottom, it will correspond to the \emph{product} of the multinomials of each of the columns of the matrix. We will also assume the typical convention that if any of the terms in the multinomial are negative, or if the sum of the terms on the bottom of the multinomial are not equal to the top, then the value of the multinomial is zero.

Without loss of generality we assume that the transition rate $\mathbf{q}_{ij}$ from the lumped state $\mathbf{s}^{[i]}$ to the lumped state $\mathbf{s}^{[j]}$ corresponds to a vertex in vertex-partition $q\in\{1,2\}$ transitioning from vertex-state $\mathcal{A}\in\mathcal{W}$ to $\mathcal{B}\in\mathcal{W}$, where $\mathcal{A}\ne\mathcal{B}$. 
To understand why, suppose that $S^{[k]}\in\mathcal{S}_i$, then for there to be a non-zero rate from $\mathbf{s}^{[i]}$ to $\mathbf{s}^{[j]}$ there must be a microstate $S^{[l]}\in\mathcal{S}_j$ such that $S^{[k]}$ and $S^{[l]}$ form a transition pair, whose transition vertex is $v$ say. We are free to use $V_q$ to label which vertex-partition cell $v$ belongs to, and we may also use the labels $\mathcal{A},\mathcal{B}\in\mathcal{W}$ to indicate what vertex-states $v$ changes from and to respectively. Moreover, any other non-zero transition rate from $\mathbf{s}^{[i]}$ to $\mathbf{s}^{[j]}$ must also correspond to a vertex in $V_q$ changing from $\mathcal{A}$ to $\mathcal{B}$, otherwise it would result in different counts of vertices of each type of vertex-state in each vertex-partition cell, i.e. a lumped state different to $\mathbf{s}^{[j]}$.
To compute \eqref{eq:qij}, for each $v\in V_q$ we can construct all possible microstates where $v$ has vertex-state $\mathcal{A}$. If in this process we specify the vertex-states of the neighbours of $v$, then we can determine the rate at which $v$ changes from $\mathcal{A}$ to $\mathcal{B}$. Summing this contribution from all possible cases yields $\mathbf{q}_{ij}$. A proof of this will be given in Section~\ref{sec:Ppartitions}. In the following paragraphs we will build up the components of this sum.

First we define some notation related to the neighbourhoods of vertices.
Let $d_p^v$ denote the number of neighbours of vertex $v$ in the $p$th vertex-partition cell. 
The degree of vertex $v$ is
\begin{equation}
    d^v:=\sum_{p=1}^Pd^v_p.
    \label{eq:dv}
\end{equation}
We represent the neighbourhood of $v$ using a non-negative, integer-valued matrix $\mathbf{n}^v\in \mathbb{Z}^{M\times P}_{\ge 0}$, which we call a \emph{neighbourhood count}. 
\begin{definition}
For a vertex $v\in V$, a \emph{neighbourhood count} is a matrix $\mathbf{n}^v\in\mathbb{Z}^{M\times P}_{\ge 0}$ such that $$\sum_{m=1}^M\mathbf{n}^v_{m,p}=d^v_p,$$
for $0\le p\le P$.
\end{definition}
Note that it follows from \eqref{eq:dv} that $\sum_{p=1}^P\sum_{m=1}^M\mathbf{n}^v_{m,p}=d^v.$
We will use a single index on this matrix to indicate a column, i.e.
$\mathbf{n}^v_p\in\mathbb{Z}^{M}_{\ge 0}$ is the $p$th column of $\mathbf{n}^v$.
The component $\mathbf{n}^v_{m,p}$ of a neighbourhood count is the number of neighbours of vertex $v$ in the $m$th vertex-state and in the $p$th vertex-partition cell.

The number of ways that we can arrange the vertex-states of the neighbours of $v$ in vertex-partition cell $p\neq q$ according to some $\mathbf{n}^v_p$, as well as the vertex-states of the other vertices in vertex-partition cell $p\neq q$ according to $\mathbf{s}_p$ is
\begin{equation}
A(\mathbf{s}_p,\mathbf{n}^v_p):=\binom{d^v_p}{\mathbf{n}^v_{p}}\binom{N_p-d^v_p}{\mathbf{s}_p-\mathbf{n}^v_p}.
\label{eq:arrangements}
\end{equation}
Recall that a vector in the bottom of the multinomial coefficient notation indicates that the elements of the vector should be in the denominator of the multinomial coefficient.
While in \eqref{eq:arrangements} we have used $d^v_p$ and $N_p$ in the top of the multinomial coefficients, we will assume that these values are actually determined from summing the vectors in the bottom. We will also assume the standard convention that a multinomial coefficient is zero if any entry is negative.

For the vertex-partition cell $q$, which contains $v$, the number of ways that we can arrange the neighbours of $v$ according to $\mathbf{n}^v_q$ is
\begin{equation*}
	A(\mathbf{s}_q-\mathbf{e}_{\mathcal{A}},\mathbf{n}^v_q)=\binom{d^v_q}{\mathbf{n}^v_{q}}\binom{N_q-1-d^v_q}{\mathbf{s}_q-\mathbf{e}_{\mathcal{A}}-\mathbf{n}^v_q},
\end{equation*}
where $\mathbf{e}_{\mathcal{A}}$ is a vector of length $M$ with a one in the entry corresponding to vertex-state $\mathcal{A}$ and zeros elsewhere. This is to account for the fact that we assumed vertex $v$ is in vertex-state $\mathcal{A}$, so there is one less $\mathcal{A}$ vertex in $\mathbf{s}_q$. 

For a given $\mathbf{n}^v$ and $p\ne q$, using \eqref{eq:arrangements} there are
\begin{equation}
A(\mathbf{s}^{[i]}_q-\mathbf{e}_{\mathcal{A}},\mathbf{n}^v_q)A(\mathbf{s}^{[i]}_p,\mathbf{n}^v_p)
\label{eq:nostatesvIq}
\end{equation}
microstates in $\mathcal{S}_i$ in which vertex $v$ in vertex partition cell $q$ is in vertex-state $\mathcal{A}$, its neighbours' vertex-states correspond to $\mathbf{n}^v$, and the total number of vertices in each vertex-state and in each vertex-partition cell corresponds to $\mathbf{s}^{[i]}$. We also need to know the rate that vertex $v$ will change from vertex-state $\mathcal{A}$ to $\mathcal{B}$.
For SVTs with affine VSTM given by \eqref{eq:affinevstm}, if vertex $v$ is in vertex-state $\mathcal{A}$ and the number of its neighbours in each of the vertex-states in each of the vertex-partition cells is given by $\mathbf{n}^v$, then it will transition from $\mathcal{A}$ to $\mathcal{B}$ with rate 
$$\zeta_0^{\mathcal{A},\mathcal{B}}+\sum_{m=1}^{M}\sum_{r=1}^{P}\zeta_m^{\mathcal{A},\mathcal{B}}\mathbf{n}^v_{m,r}.$$

Recall that we assumed, without loss of generality, that $\mathbf{q}_{ij}$ corresponds to a vertex in vertex-partition cell $V_q$ changing from vertex-state $\mathcal{A}\in\mathcal{W}$ to $\mathcal{B}\in\mathcal{W}$.
We can now compute $\mathbf{q}_{ij}$ by summing over all feasible realisations of the matrix $\mathbf{n}^v$ and vertices in $V_q$, which yields
\begin{equation}
\mathbf{q}_{ij}=\frac{1}{{N\choose \mathbf{s}^{[i]}}}\sum_{v\in V_q}\sum_{\mathbf{n}^v_1|d_1^v}\sum_{\mathbf{n}^v_2|d_2^v}
\left(\zeta_0^{\mathcal{A},\mathcal{B}}+\sum_{m=1}^{M}\sum_{r=1}^{P}\zeta_m^{\mathcal{A},\mathcal{B}}\mathbf{n}^v_{m,r}\right)
A(\mathbf{s}^{[i]}_q-\mathbf{e}_{\mathcal{A}},\mathbf{n}^v_q)A(\mathbf{s}^{[i]}_p,\mathbf{n}^v_p).
\label{eq:Rij2}
\end{equation}
In this equation, the sums over $\mathbf{n}^v_1|d_1^v$ and $\mathbf{n}^v_2|d_2^v$ specify the $P=2$ columns of $\mathbf{n}^v$.

It seems that we have swapped one difficult sum, \eqref{eq:qij}, for another, \eqref{eq:Rij2}. However, it turns out that the sum over vertex neighbourhoods can be simplified, and we will illustrate how this can be done in the next section.

\subsection{Simplified form of the lumped generator}
\label{sec:simplifiedgen}
Our next step is to simplify \eqref{eq:Rij2}.
Crucial to achieving this is a generalisation of the Vandermonde identity and what we call the ``sum-product property'', which for arbitrary $v$ can be used to obtain
\begin{equation}
{N\choose \mathbf{s}}=\sum_{\mathbf{n}^v_1|d_1^v}\sum_{\mathbf{n}^v_2|d_2^v}
A(\mathbf{s}_1,\mathbf{n}^v_1)A(\mathbf{s}_2,\mathbf{n}^v_2).
\label{eq:VanDerMondeP2}
\end{equation}
A generalisation of this will be presented in Section~\ref{sec:Ppartitions} and further details can be found in the Supplementary Materials.
First consider the $\zeta_0^{\mathcal{A},\mathcal{B}}$ term in \eqref{eq:Rij2}, then using  
\eqref{eq:VanDerMondeP2} we find
\begin{align*}
	&\frac{1}{{N\choose \mathbf{s}^{[i]}}}\sum_{v\in V_q}\sum_{\mathbf{n}^v_1|d_1^v}\sum_{\mathbf{n}^v_2|d_2^v}
	\zeta_0^{\mathcal{A},\mathcal{B}}
	A(\mathbf{s}^{[i]}_q-\mathbf{e}_{\mathcal{A}},\mathbf{n}^v_q)A(\mathbf{s}^{[i]}_p,\mathbf{n}^v_p)\\
	&=\frac{1}{{N\choose \mathbf{s}^{[i]}}}\sum_{v\in V_q}
	\zeta_0^{\mathcal{A},\mathcal{B}}
	\binom{N_q-1}{\mathbf{s}^{[i]}_q-\mathbf{e}_{\mathcal{A}}}
	\binom{N_p}{\mathbf{s}^{[i]}_p}
	=\sum_{v\in V_q}\zeta_0^{\mathcal{A},\mathcal{B}}\frac{\mathbf{s}_{\mathcal{A},q}^{[i]}}{N_q}
	=\zeta_0^{\mathcal{A},\mathcal{B}}\mathbf{s}_{\mathcal{A},q}^{[i]},
\end{align*}
where we have used $\mathcal{A}$ as an index in $\mathcal{W}$.
This is what we might have expected: each microstate in $\mathbf{s}^{[i]}$ has $\mathbf{s}_{\mathcal{A},q}^{[i]}$ vertices in vertex partition cell $q$ in vertex-state $\mathcal{A}$, so there will be a factor of $\zeta_0^{\mathcal{A},\mathcal{B}}$ for each of these.

Now consider the $\zeta_m^{\mathcal{A},\mathcal{B}}$ terms in \eqref{eq:Rij2}. The sums over $m$ and $r$ in \eqref{eq:Rij2} do not depend on any of the others sums, so we can move these to the outside. 
Note that
\begin{align*}
	\mathbf{n}^v_{m,p}A(\mathbf{s}_p,\mathbf{n}^v_p)&=
	\mathbf{n}^v_{m,p}\binom{d^v_p}{\mathbf{n}^v_p}\binom{N_p-d^v_p}{\mathbf{s}_p-\mathbf{n}^v_p}\\
	&=d^v_p\binom{d^v_p-1}{\mathbf{n}^v_p-\mathbf{e}_m}\binom{N_p-1-(d^v_p-1)}{\mathbf{s}_p-\mathbf{e}_m-(\mathbf{n}^v_p-\mathbf{e}_m)}\\
	&=d^v_pA(\mathbf{s}_p-\mathbf{e}_m,\mathbf{n}^v_p-\mathbf{e}_m).
\end{align*}
For the case where $r=q$, it follows that
\begin{align*}
	&\sum_{\mathbf{n}^v_1|d_1^v}\sum_{\mathbf{n}^v_2|d_2^v}\zeta_m^{\mathcal{A},\mathcal{B}}\mathbf{n}^v_{m,q}
	A(\mathbf{s}^{[i]}_q-\mathbf{e}_{\mathcal{A}},\mathbf{n}^v_q)A(\mathbf{s}^{[i]}_p,\mathbf{n}^v_p)\\
	&=\zeta_m^{\mathcal{A},\mathcal{B}}d_q^v\sum_{\mathbf{n}^v_1|d_1^v}\sum_{\mathbf{n}^v_2|d_2^v}
	A(\mathbf{s}^{[i]}_q-\mathbf{e}_{\mathcal{A}}-\mathbf{e}_m,\mathbf{n}^v_q-\mathbf{e}_m)A(\mathbf{s}^{[i]}_p,\mathbf{n}^v_p)\\
	&=\zeta_m^{\mathcal{A},\mathcal{B}}d_q^v
	\binom{N_q-2}{\mathbf{s}^{[i]}_q-\mathbf{e}_{\mathcal{A}}-\mathbf{e}_m}
	\binom{N_p}{\mathbf{s}^{[i]}_p}\\
\end{align*}
where we have used \eqref{eq:VanDerMondeP2} to obtain the third line. 
Similarly, for the case $r=p$ we have 
\begin{align*}
	\sum_{\mathbf{n}^v_1|d_1^v}\sum_{\mathbf{n}^v_2|d_2^v}\zeta_m^{\mathcal{A},\mathcal{B}}\mathbf{n}^v_{m,p}
	A(\mathbf{s}^{[i]}_q-\mathbf{e}_{\mathcal{A}},\mathbf{n}^v_q)A(\mathbf{s}^{[i]}_p,\mathbf{n}^v_p)
	=\zeta_m^{\mathcal{A},\mathcal{B}}d_p^v
	\binom{N_q-1}{\mathbf{s}^{[i]}_q-\mathbf{e}_{\mathcal{A}}}
	\binom{N_p-1}{\mathbf{s}^{[i]}_p-\mathbf{e}_m}.
\end{align*}
Thus for the $\zeta_m^{\mathcal{A},\mathcal{B}}$ terms in \eqref{eq:Rij2} we can cancel multinomial terms to get
\begin{align*}
	&\frac{1}{{N\choose \mathbf{s}^{[i]}}}\sum_{v\in V_q}\sum_{m=1}^M
	\zeta_m^{\mathcal{A},\mathcal{B}}
	\left[
	d^v_q\binom{N_q-2}{\mathbf{s}^{[i]}_q-\mathbf{e}_{\mathcal{A}}-\mathbf{e}_m}
	\binom{N_p}{\mathbf{s}^{[i]}_p}
	+d^v_p\binom{N_q-1}{\mathbf{s}^{[i]}_q-\mathbf{e}_{\mathcal{A}}}
	\binom{N_p-1}{\mathbf{s}^{[i]}_p-\mathbf{e}_m}
	\right]\\
	&=\sum_{v\in V_q}
	\sum_{m=1}^M\zeta_m^{\mathcal{A},\mathcal{B}}\left[d_q^v
	\frac{\mathbf{s}^{[i]}_{\mathcal{A},q}\left(\mathbf{s}^{[i]}_{m,q}-\delta_{\mathcal{A},\mathcal{W}_m}\right)}{N_q(N_q-1)}
	+d_{p}^v
	\frac{\mathbf{s}^{[i]}_{\mathcal{A},q}\mathbf{s}^{[i]}_{m,p}}{N_qN_{p}}\right],
\end{align*}
where recall that $\delta_{\mathcal{A},\mathcal{B}}$ is the Kronecker delta function.
Using these simplifications, the expression for $\mathbf{q}_{ij}$ becomes
\begin{align*}
	\mathbf{q}_{ij}=\mathbf{s}_{\mathcal{A},q}^{[i]}\left\{\zeta_0^{\mathcal{A},\mathcal{B}}+\frac{1}{N_q}
	\sum_{v\in V_q}
	\sum_{m=1}^M\zeta_m^{\mathcal{A},\mathcal{B}}\left[d_q^v
	\frac{\left(\mathbf{s}^{[i]}_{m,q}-\delta_{\mathcal{A},\mathcal{W}_m}\right)}{(N_q-1)}
	+d_{p}^v
	\frac{\mathbf{s}^{[i]}_{m,p}}{N_{p}}\right]\right\}.
\end{align*}
Note that this formula enables us to compute $\mathbf{q}$ without direct reference to $\mathbf{Q}$.

\section{The lumped generator for arbitrary finite vertex-partitions}
\label{sec:Ppartitions}
We can now generalise relatively easily to arbitrary vertex partitions. We consider a vertex partition $\Pi=\{V_1,\dots,V_P\}$ with $V_p\cap V_q=\emptyset$ for $p\ne q$ and $\cup_{p}V_p=V$. We then have $\mathbf{s}\in\mathbb{Z}^{M\times P}_{\ge0}$, whose $m,p$th entry, $\mathbf{s}_{m,p}$, is the number of vertices in vertex-state $\mathcal{W}_m$ in the vertex-partition cell $V_p$. The total number of microstates is
$$M^N=\sum_{\mathbf{s}}{N \choose \mathbf{s}}:=\sum_{\mathbf{s}_1|N_1}\sum_{\mathbf{s}_2|N_2}\cdots\sum_{\mathbf{s}_P|N_P}\prod_{p=1}^P{N_p\choose \mathbf{s}_p},$$
which can be obtained from application of the multinomial theorem and interchanging the product of sums with a sum of products (see the Supplementary Information for further details).
The number of lumped states is
\begin{equation}
n=\prod_{p=1}^{P}{N_p+M-1\choose N_p}.
\label{eq:lumpedn}
\end{equation}
Again we assume, without loss of generality, that a transition from the lumped state $\mathbf{s}^{[i]}$ to the lumped state $\mathbf{s}^{[j]}\ne\mathbf{s}^{[i]}$ corresponds to a vertex in vertex-partition cell $V_q$ transitioning from vertex-state $\mathcal{A}$ to $\mathcal{B}$.

Our initial goal is to write \eqref{eq:qij} in terms of a sum over vertices and possible neighbourhood counts. In order to prove this, we must first consider some properties of neighbourhood counts.
Since we sum over neighbourhood counts, we need to be clear about which correspond to arrangements of vertex-states in microstates in the correct lumped partition cell, and that only these contribute non-zero values to the summation. To this end, we say a neighbourhood count $\mathbf{n}^v$ is \emph{realisable} if there is at least one microstate in which the number of neighbours of vertex $v$ in each vertex-partition corresponds to $\mathbf{n}^v$.
\begin{definition}
A neighbourhood count $\mathbf{n}^v$ is \emph{realisable} in $\mathcal{S}_i\subset\mathcal{S}$ if there is a microstate $S^{[k]}\in\mathcal{S}_i$ in which the number of neighbours of vertex $v$ in vertex-state $\mathcal{W}_m\in\mathcal{W}$ and vertex-partition cell $V_p\in\Pi_V$ is $\mathbf{n}^v_{m,p}$.
\end{definition}
There is a simple condition that ensures a neighbourhood count is realisable.
\begin{lemma}
For $v\in V$, a neighbourhood count $\mathbf{n}^v$ is realisable in $\mathcal{S}_i\subset\mathcal{S}$ if and only if $\mathbf{n}^v_{m,p}\le\mathbf{s}^{[i]}_{m,p}$ for all $m,p$.    
\label{lem:nvrealisableiflts}
\end{lemma}
\begin{proof}
Suppose that $\mathbf{n}^v$ is realisable in   $\mathcal{S}_i\subset\mathcal{S}$, then there is a microstate $S^{[k]}\in\mathcal{S}_i$ such that the number of neighbours of vertex $v$ in vertex-state $\mathcal{W}_m\in\mathcal{W}$ and vertex-partition cell $V_p\in\Pi_V$ is $\mathbf{n}^v_{m,p}$. Since $\mathbf{s}^{[i]}_{m,p}$ is the number of vertices in vertex-state $\mathcal{W}_m\in\mathcal{W}$ and vertex-partition cell $V_p\in\Pi_V$, we must have $\mathbf{n}^v_{m,p}\le\mathbf{s}^{[i]}_{m,p}$ for all $m,p$.

Suppose that $\mathbf{n}^v_{m,p}\le\mathbf{s}^{[i]}_{m,p}$ for all $m,p$. Then for each partition $p$, we can assign a vertex-state to each of the neighbours of $v$ such that the total number of neighbours of $v$ in vertex-state $\mathcal{W}_m\in\mathcal{W}$ and vertex-partition cell $V_p\in\Pi_V$ is $\mathbf{n}^v_{m,p}$. Similarly we can assign vertex-states to each of the remaining vertices (including $v$) in each vertex-partition cell such that the number of them in vertex-state $\mathcal{W}_m\in\mathcal{W}$ and vertex-partition cell $V_p\in\Pi_V$ is $\mathbf{s}^{[i]}_{m,p}-\mathbf{n}^v_{m,p}\ge0$. Thus the total number of vertices in vertex-state $\mathcal{W}_m\in\mathcal{W}$ and vertex-partition cell $V_p\in\Pi_V$ is $\mathbf{s}^{[i]}_{m,p}$, and hence this assignment of vertex-states corresponds to a microstate in $\mathcal{S}_i$.
\end{proof}

Importantly, when we sum over neighbourhood counts, only those that are realisable contribute to the sum.
\begin{corollary}
For $v\in V$, a neighbourhood count $\mathbf{n}^v$ is realisable in $\mathcal{S}_i\subset\mathcal{S}$ if and only if
\begin{equation}
\prod_{p=1}^{P}A(\mathbf{s}^{[i]}_p,\mathbf{n}^v_p)>0.
\label{eq:prodApositive}
\end{equation}
\label{col:nvrealisablepositiveprod}
\end{corollary}
\begin{proof}
Suppose that for $v\in V$, a neighbourhood count $\mathbf{n}^v$ is realisable in $\mathcal{S}_i\subset\mathcal{S}$, then by Lemma~\ref{lem:nvrealisableiflts} $\mathbf{s}^{[i]}_{m,p}-\mathbf{n}^v_{m,p}\ge0$ for all $m,p$ and hence \eqref{eq:prodApositive} is true.
Conversely, if \eqref{eq:prodApositive} is true then each term in the product is positive and hence $\mathbf{n}^v_{m,p}\le\mathbf{s}^{[i]}_{m,p}$ for all $m,p$. Then by Lemma~\ref{lem:nvrealisableiflts} $\mathbf{n}^v$ is realisable in $\mathcal{S}_i\subset\mathcal{S}$.
\end{proof}

We may now prove that lumped rate \eqref{eq:qij} can be written in terms of a sum over vertices and possible neighbourhood counts. In the following lemma, we use $\mathbf{1}_P$ to denote a column vector with $P$ entries, each 1.
\begin{lemma}
Let $\mathcal{S}$ be the microstate-space of a homogeneous SVT with affine VSTM on a network with vertex set $V$ and let $\Pi_V=\{V_1,V_2,\dots,V_{P}\}$ be a partition of $V$. Suppose that the corresponding vertex-partition lumping macrostate-space is $(\mathbf{s}^{[1]},\mathbf{s}^{[2]},\dots,\mathbf{s}^{[n]})$ (Definition~\ref{def:vpppms}) and the partition of microstates is $\Pi_{\mathcal{S}}=\{\mathcal{S}_1,\mathcal{S}_2,\dots,\mathcal{S}_n\}$ (Definition~\ref{def:vpal}). If a transition from $\mathbf{s}^{[i]}$ to $\mathbf{s}^{[j]}$ corresponds to a vertex $v\in V_q$ changing from vertex-state $\mathcal{A}$ to vertex-state $\mathcal{B}$, then
\begin{equation}
\sum_{S^{[k]}\in \mathcal{S}_i}\sum_{S^{[l]}\in \mathcal{S}_j}\mathbf{Q}_{kl}=\sum_{v\in V_q}\sum_{\mathbf{n}^v\mid d^v}\mathbf{R}_{\mathcal{A},\mathcal{B}}(\mathbf{n}^v\mathbf{1}_P)
\prod_{p=1}^{P}A(\mathbf{s}^{[i]}_p-\delta_{p,q}\mathbf{e}_{\mathcal{A}},\mathbf{n}^v_p)
\label{eq:count-eq}
\end{equation}
\label{lm:counting}
\end{lemma}
\begin{proof}
We construct a surjective map from transition pairs on the left of \eqref{eq:count-eq} to feasible neighbourhood counts $\mathbf{n}^v$ on the right of \eqref{eq:count-eq}, and show that the multinomial terms on the right account for the many-to-one multiplicity of the mapping.
For $S^{[k]}\in \mathcal{S}_i$ and $S^{[l]}\in \mathcal{S}_j$, $\mathbf{Q}_{kl}$ can only be non-zero if $S^{[k]}\stackrel{v}{\sim} S^{[l]}$, which identifies a unique vertex $v$ and by assumption $v\in V_q$. Then $\mathbf{Q}_{kl}=\mathbf{R}_{S^{[k]}(v),S^{[l]}(v)}(n^{[k]}(v))$. Furthermore, the microstate $S^{[k]}$ corresponds to a unique neighbourhood count $\mathbf{n}^v$, which is evidently realisable. 
Since $n^{[k]}(v)$ is derived from $S^{[k]}$, we have $n^{[k]}(v)=\mathbf{n}^v\mathbf{1}_P$.
It is assumed that a transition from $\mathbf{s}^{[i]}$ to $\mathbf{s}^{[j]}$ corresponds to a vertex $v\in V_q$ changing from vertex-state $\mathcal{A}$ to vertex-state $\mathcal{B}$, thus we have $S^{[k]}(v)=\mathcal{A}$ and $S^{[l]}(v)=\mathcal{B}$. 
Consequently, each transition pair in the summation on the left of \eqref{eq:count-eq} corresponds to a unique pair $v\in V_q$ and $\mathbf{n}^v|d^v$ in the summation on the right with rate $\mathbf{R}_{\mathcal{A},\mathcal{B}}(\mathbf{n}^v\mathbf{1}_P)$. Since $\mathbf{n}^v$ is realisable and the vertex-state of $v\in V_q$ is $\mathcal{A}$, $\prod^P_{p=1}A(\mathbf{s}^{[i]}_p-\delta_{p,q}\mathbf{e}_{\mathcal{A}},\mathbf{n}^v_p)>0$.

For a given $v\in V_q$ and neighbourhood count $\mathbf{n}_p^v$, when $p\ne q$ there are 
\begin{equation*}
A(\mathbf{s}^{[i]}_p,\mathbf{n}^v_p)
=\binom{d^v_p}{\mathbf{n}^v_p}\binom{N_p-d^v_p}{\mathbf{s}^{[i]}_p-\mathbf{n}^v_p} 
\end{equation*}
ways of arranging the vertex-states of the neighbours of $v$ in vertex-partition cell $p$ according to $\mathbf{n}^v_p$, as well as the vertex-states of the other vertices in vertex-partition cell $p$ according to $\mathbf{s}^{[i]}_p$. 
Similarly, for the case $p=q$ we have 
\begin{equation*}
A(\mathbf{s}^{[i]}_p-\mathbf{e}_{\mathcal{A}},\mathbf{n}^v_p)
=\binom{d^v_p}{\mathbf{n}^v_p}\binom{N_p-1-d^v_p}{\mathbf{s}^{[i]}_p-\mathbf{e}_{\mathcal{A}}-\mathbf{n}^v_p}.
\end{equation*}
Thus using Corollary~\ref{col:nvrealisablepositiveprod}, there are 
\begin{equation}
\prod_{p=1}^{P}A(\mathbf{s}^{[i]}_p-\delta_{p,q}\mathbf{e}_{\mathcal{A}},\mathbf{n}^v_p)
\label{eq:dvncombinations}
\end{equation}
microstates in $\mathcal{S}_i$ in which the vertex-state of $v$ is $\mathcal{A}$ and its neighbourhood count is $\mathbf{n}^v$. These microstates form transition pairs with the corresponding states in $\mathcal{S}_j$ in which the vertex-state of $v$ is $\mathcal{B}$.
Thus the number of transition pairs on the left of \eqref{eq:count-eq} associated with each $v\in V_q$ and realisable neighbourhood count $\mathbf{n}^v$ in the summation on the right is given by \eqref{eq:dvncombinations}.
\end{proof}

We saw in Section~\ref{sec:simplifiedgen} that we needed a generalisation of the Vandermonde identity. This can be stated as
\begin{lemma}
Let $\mathbf{s}$ be a lumped state and $\mathbf{n}^v$ a neighbourhood count of a vertex $v\in V$, then
\begin{equation}
\binom{N}{\mathbf{s}}=
\sum_{\mathbf{n}^v\mid d^v}\prod_{p=1}^PA(\mathbf{s}_p,\mathbf{n}^v_p).
\end{equation}
\label{lem:vandermonde}
\end{lemma}
A detailed proof can be found in the Supplementary Materials.

We now present the main result of the paper.

\begin{theorem}
Let $\mathcal{S}$ be the state-space of a homogeneous SVT with affine VSTM on a network with vertex set $V$ and let $\Pi_V=\{\Pi_1,\Pi_2,\dots,\Pi_{P}\}$ be a partition of $V$. 
Suppose that $\mathbf{q}=\mathbf{DQC}$ is the lumped infinitesimal generator corresponding to the vertex-partition lumping with macrostate-space
$(\mathbf{s}^{[1]},\mathbf{s}^{[2]},\dots,\mathbf{s}^{[n]})$. 
If a transition from $\mathbf{s}^{[i]}$ to $\mathbf{s}^{[j]}$ corresponds to a vertex $v\in V_q$ changing from vertex-state $\mathcal{A}$ to vertex-state $\mathcal{B}$, then
\begin{align}
\mathbf{q}_{ij}=\mathbf{s}_{\mathcal{A},q}^{[i]}\left\{
\zeta_0^{\mathcal{A},\mathcal{B}}+\frac{1}{N_q}\sum_{r=1}^P
\left(\frac{\sum_{v\in V_q}d^v_r}{N_r-\delta_{q,r}}\right)
\left[\sum_{m=1}^{M}\zeta_m^{\mathcal{A},\mathcal{B}}
\left(\mathbf{s}_{m,r}^{[i]}-\delta_{\mathcal{A},\mathcal{W}_m}\delta_{q,r}\right)
\right]\right\}.\label{eq:qijgeneral}
\end{align}
\label{thm:genrates}
\end{theorem}

\begin{proof}
From \eqref{eq:qij} we have 
$$\mathbf{q}_{ij}=\frac{1}{\binom{N}{\mathbf{s}^{[i]}}}\sum_{S^{[k]}\in \Pi_i}\sum_{S^{[l]}\in \Pi_j}\mathbf{Q}_{km},$$
since the number of states in $\Pi_i$ is $|\Pi_i|=\binom{N}{\mathbf{s}^{[i]}}.$
Using Lemma~\ref{lm:counting} and \eqref{eq:affinevstm} it follows that
\begin{equation}
\mathbf{q}_{ij}=\frac{1}{{N\choose \mathbf{s}^{[i]}}}\sum_{v\in V_q}\sum_{\mathbf{n}^v|d^v}
\left(\zeta_0^{\mathcal{A},\mathcal{B}}+\sum_{m=1}^{M}\sum_{r=1}^{P}\zeta_m^{\mathcal{A},\mathcal{B}}\mathbf{n}^v_{m,r}\right)
\prod_{p=1}^PA(\mathbf{s}^{[i]}_p-\delta_{p,q}\mathbf{e}_{\mathcal{A}},\mathbf{n}^v_p)
\label{eq:Rijexpansion}
\end{equation}

We will deal with the $\zeta_0^{\mathcal{A},\mathcal{B}}$ and $\zeta_m^{\mathcal{A},\mathcal{B}}n_m$ terms separately. From Lemma~\ref{lem:vandermonde}, the sums around the constant term $\zeta_0^{\mathcal{A},\mathcal{B}}$ are
\begin{align}
    \frac{1}{\binom{N}{\mathbf{s}^{[i]}}}\sum_{v\in V_q}\sum_{\mathbf{n}^v\mid d^v}
\zeta_0^{\mathcal{A},\mathcal{B}}
\prod^P_{p=1}A(\mathbf{s}^{[i]}_p-\delta_{p,q}\mathbf{e}_{\mathcal{A}},\mathbf{n}^v_p)
&=\frac{1}{\binom{N}{\mathbf{s}^{[i]}}}\sum_{v\in V_q}\zeta_0^{\mathcal{A},\mathcal{B}}
\prod_{p=1}^P\binom{N_p-\delta_{p,q}}{\mathbf{s}^{[i]}_p-\delta_{p,q}\mathbf{e}_{\mathcal{A}}},\nonumber\\
&=\zeta_0^{\mathcal{A},\mathcal{B}}\mathbf{s}^{[i]}_{\mathcal{A},q}.\label{eq:c0}
\end{align}
In \eqref{eq:Rijexpansion} we are able to do the sums over $m$ and $r$ after we sum over $v$ and $\mathbf{n}^v|d^v$, since they do not determine $\mathbf{n}^v$. Thus we can focus on an individual term $\zeta_m^{\mathcal{A},\mathcal{B}}\mathbf{n}_{m,r}$; using Lemma~\ref{lem:vandermonde} we have
\begin{align*}
\sum_{\mathbf{n}^v\mid d^v}
\mathbf{n}^v_{m,r}
\prod^P_{p=1}A(\mathbf{s}^{[i]}_{p}-\delta_{p,q}\mathbf{e}_{\mathcal{A}},\mathbf{n}^v_p)
=d^v_r
\prod^P_{p=1}\binom{N_p-\delta_{p,q}-\delta_{p,r}}{\mathbf{s}^{[i]}_p-\delta_{p,q}\mathbf{e}_{\mathcal{A}}-\delta_{p,r}\mathbf{e}_{m}}.
\end{align*}
We then find
\begin{align}
    &\frac{1}{\binom{N}{\mathbf{s}^{[i]}}}\sum_{v\in V_q}\sum_{\mathbf{n}^v\mid d^v}\sum_{m=1}^M\sum_{r=1}^P
\zeta_m^{\mathcal{A},\mathcal{B}}\mathbf{n}^v_{m,r}
\prod^P_{p=1}A(\mathbf{s}^{[i]}_p-\delta_{p,q}\mathbf{e}_{\mathbf{A}},\mathbf{n}^v_p)\nonumber\\
=&\frac{1}{\binom{N}{\mathbf{s}^{[i]}}}
\sum_{v\in V_q}\sum_{m=1}^M\sum_{r=1}^P
\zeta_m^{\mathcal{A},\mathcal{B}}d^v_r
\prod^P_{p=1}\binom{N_p-\delta_{p,q}-\delta_{p,r}}{\mathbf{s}^{[i]}_p-\delta_{p,q}\mathbf{e}_{\mathcal{A}}-\delta_{p,r}\mathbf{e}_{m}}\nonumber\\
=&\sum_{v\in V_q}\sum_{m=1}^M\sum_{r=1}^P
\zeta_m^{\mathcal{A},\mathcal{B}}d^v_r
\frac{\mathbf{s}^{[i]}_{\mathcal{A},q}(\mathbf{s}^{[i]}_{m,r}-\delta_{\mathcal{A},\mathcal{W}_m}\delta_{q,r})}{N_q(N_r-\delta_{q,r})}
.\label{eq:ck}
\end{align}
Substituting \eqref{eq:c0} and \eqref{eq:ck} into \eqref{eq:Rijexpansion}, after some rearranging yields \eqref{eq:qijgeneral}.
\end{proof}

\section{Applications of Theorem~\ref{thm:genrates}}
\label{sec:applications}
We now describe two special cases of Theorem~\ref{thm:genrates} and illustrate its application on a bipartite network.

\subsection{Recovering the infinitesimal generator}
In the case where each vertex is in its own partition, we expect to recover the full infinitesimal generator $\mathbf{Q}$. Recall that \eqref{eq:qijgeneral} corresponds to a vertex $v$ in vertex partition cell $q$ changing from vertex-state $\mathcal{A}$ to $\mathcal{B}$. For each $p$ we have $N_p=1$, and since we have assumed the network is simple, i.e. there are no self-loops or multiple edges, it follows that $d_q^v=0$. 
We also have $d_p^v=1$ if vertex $v$ and the vertex in cell $p$ are neighbours, and $d_p^v=0$ if they are not. Since vertex $v$ is the only vertex in cell $q$ and it is in vertex-state $\mathcal{A}$, we have $\mathbf{s}_{\mathcal{A},q}^{[i]}=1$. We also have
$$\sum_{p\ne q}\left(\sum_{v\in V_q}\frac{d_{p}^v}{N_{p}}\right)
\left[\sum_{m=1}^M\zeta_m^{\mathcal{A},\mathcal{B}}\mathbf{s}_{m,p}^{[i]}\right]
=\sum_{m=1}^M\zeta_m^{\mathcal{A},\mathcal{B}} n^v_m,$$
where $n^v_m$ is the number of neighbours of $v$ that are in vertex-state $\mathcal{W}_m$. This follows from the fact that the sum over $p$ is effectively a sum over neighbours, since $d_p^v$ is zero otherwise, and $\mathbf{s}_{m,p}^{[i]}=1$ if the neighbouring vertex in cell $p$ is in vertex-state $\mathcal{W}_m$ and zero otherwise.
Thus when each vertex is in its own partition, \eqref{eq:qijgeneral} reduces to
$$\mathbf{q}_{ij}=\zeta_0^{\mathcal{A},\mathcal{B}}+\sum_{m=1}^M\zeta_m^{\mathcal{A},\mathcal{B}} n^v_m,$$
i.e. our definition of $\mathbf{Q}_{ij}$.

\subsection{Single vertex partition: population model} \label{subsec:single partition}
In the case $P=1$, Theorem \ref{thm:genrates} reduces to our result published in \cite{ward2022micro}. In this case, the lumped state can be represented by a vector, i.e. the number of vertices in each of the possible vertex-states, and so we will not use a bold font to represent lumped states. In the $P=1$ case, \eqref{eq:qijgeneral} becomes
$$
\mathbf{q}_{ij}=s_{\mathcal{A}}^{[i]}\left\{\zeta_0^{\mathcal{A},\mathcal{B}}+\frac{z}{N-1}
\left[\sum_{m=1}^{M}\zeta_{m}^{\mathcal{A},\mathcal{B}}\left(s_{m}^{[i]}-\delta_{\mathcal{A},\mathcal{W}_m}\right)\right]\right\},
$$
where $z=\sum_{v\in V}d^v/N$ is the mean degree of the network and $s_m^{[i]}$ ($s_{\mathcal{A}}^{[i]}$) is the number of vertices in lumped state $i$ that are in vertex-state $\mathcal{W}_m$ ($\mathcal{A}$). This corresponds to the population model equation derived in \cite{ward2022micro}.

\subsection{SIS epidemic on a complete bipartite graph}
For the case of SIS dynamics, the lumped states $\mathbf{s}$ are matrices of size $2\times P$, where $m=1$ corresponds to susceptible vertices and $m=2$ corresponds to infected vertices. 
For the SIS model, we have $\zeta_2^{1,2}=\beta$, $\zeta_0^{2,1}=\gamma$ and all other $\zeta_m^{\mathcal{A},\mathcal{B}}=0$. Thus if 
the transition from $\mathbf{s}^{[i]}$ to $\mathbf{s}^{[j]}$ 
corresponds to an infection in the $q$th partition, then
$$\mathbf{q}_{ij}
=\beta\frac{\mathbf{s}^{[i]}_{1,q}}{N_q}
\left\{\left(\sum_{v\in V_q}\frac{d_q^v}{N_q-1}\right)\mathbf{s}^{[i]}_{2,q}
+\sum_{p\ne q}\left(\sum_{v\in V_q}\frac{d_{p}^v}{N_{p}}\right)\mathbf{s}^{[i]}_{2,p}
\right\}.$$
Similarly, if the transition from $\mathbf{s}^{[i]}$ to $\mathbf{s}^{[j]}$ 
corresponds to a recovery in the $q$th partition, then
\begin{equation}
\mathbf{q}_{ij}
=\gamma \mathbf{s}^{[i]}_{2,q}.
\label{eq:SISrecovery}
\end{equation}
Note that this is the exact recovery rate for state $\mathbf{s}^{[i]}$, regardless of the choice of vertex-partition, since there are $\mathbf{s}^{[i]}_{2,q}$ infected vertices in $\mathbf{s}^{[i]}$.

We'll now use this to write out the lumped Markov chain equations for the SIS model on a complete bipartite graph, where the vertex-partition corresponds to the bipartite partition. For $N_1\ne N_2$, the automorphism group of a complete bipartite graph is $S_{N_1}\times S_{N_2}$, i.e. all possible pairs of permutations consisting of a permutation of the vertices in $V_1$ and a permutation of the vertices in $V_2$. Thus the microstates in the exact lumping correspond to all possible counts of the number of infected vertices in each of the vertex partition cells, which is exactly what the approximate lumping uses. Consequently, we expect the approximate lumping formula to recover the exact lumping. We have $P=2$, so the lumped states $\mathbf{s}$ are two-by-two matrices. For $v\in V_1$ we have $d_1^v=0$ and $d_2^v=N_2$; similarly for $v\in V_2$ we have $d_1^v=N_1$ and $d_2^v=0$.
Note that 
$$
\sum_{v\in V_q}d_q^v=0, \quad \mbox{ and } \quad \sum_{v\in V_q}d_{p}^v=N_qN_{p},
$$ 
where $p$ is the alternative partition to $q$. Thus if the transition from $\mathbf{s}^{[i]}$ to $\mathbf{s}^{[j]}$ 
corresponds to an infection in the $q$th partition, then
$$\mathbf{q}_{ij}
=\beta \mathbf{s}^{[i]}_{1,q}\mathbf{s}^{[i]}_{2,p}.
$$
In state $\mathbf{s}^{[i]}$, there are $\mathbf{s}^{[i]}_{1,q}$ susceptible vertices in $V_q$ and each of these has $\mathbf{s}^{[i]}_{2,p}$ infected neighbours, thus we obtain the exact lumped transition rate. The recovery rates \eqref{eq:SISrecovery} are also exact.

We can also write out the lumped master equation for this case. Note that for the lumped state $\mathbf{s}$, we have $\mathbf{s}_{1,1}+\mathbf{s}_{2,1}=N_1$ and $\mathbf{s}_{1,2}+\mathbf{s}_{2,2}=N_2$, thus we can write $\mathbf{s}$ in terms of just two numbers, $k_1=\mathbf{s}_{2,1}$ and $k_2=\mathbf{s}_{2,2}$, so $k_p$ is the number of infected vertices in the $p$th partition. Thus we write the probability of being in the lumped state $\mathbf{s}$ as $x_{k_1,k_2}$, i.e. the probability of there being $k_1$ infected nodes in partition cell 1 and $k_2$ infected nodes in partition cell 2. 
There are four possible transitions, corresponding to an infection or a recovery in each of the two vertex-partition cells.
Consequently, if we assume that $x_{k_1,k_2}=0$ for $k_1,k_2<0$ and $k_1,k_2>N$ then summing over the four possible transitions yields
\begin{align*}
	\dot{x}_{k_1,k_2}=&\beta(N_1-k_1+1)k_2x_{k_1-1,k_2}+\beta k_1(N_2-k_2+1)x_{k_1,k_2-1}\\
	&+\gamma(k_1+1)x_{k_1+1,k_2}+\gamma(k_2+1)x_{k_1,k_2+1}\\
	&-\left[\beta(N_1-k_1)k_2+\beta k_1(N_2-k_2)+\gamma(k_1+k_2)\right]x_{k_1,k_2}.
\end{align*}

We can simplify the bipartite graph case further and assume that $N_2=1$ (hence $N_1=N-1$), which corresponds to a star graph. In this case, $k_2$ can only be 0 or 1, so we have
\begin{align*}
	\dot{x}_{k_1,0}=&\gamma(k_1+1)x_{k_1+1,0}+\gamma x_{k_1,1}
	-\left[\beta k_1+\gamma k_1)\right]x_{k_1,0},\\
	\dot{x}_{k_1,1}=&\beta(N-k_1)x_{k_1-1,1}+\beta k_1x_{k_1,0}
	+\gamma(k_1+1)x_{k_1+1,1}\\
	&-\left[\beta(N-1-k_1)+\gamma(k_1+1)\right]x_{k_1,1}.
\end{align*}

\subsection{General procedure for applying Theorem~\ref{thm:genrates}}
\label{subsec:application_lumping}
Now, we briefly present the wide applicability of Theorem~\ref{thm:genrates}, explaining the procedure by which it can be applied to a broad range of node dynamics and an arbitrary choice of the vertex-set partition. Concerning node dynamics, the following processes can be handled among many others.
\begin{itemize}
	\item Beyond the widely used SIS and SIR dynamics one can use SEIR dynamic when the exposed compartment $E$ is also taken into account, as well as all other similar variants. For example, introducing a vaccination state $V$ and applying a contact tracing state $T$ lead to SIVS, SITR etc. Considering the parallel propagation of two infectious diseases, more complicated models also fall within our framework.	
	\item Information spread on networks leads to many different node-dynamics. One of them is rumour spreading when node states are ignorant ($I$), spreader ($S$) and stifler ($R$), which resembles SIR epidemic but where the transition from $S$ to $R$ depends also on the number of neighbours in the $S$ and $R$ states. Information spread is also modelled by using the node states: ``Unknown'', the individual has not yet come into contact with the information, ``Known'', the individual has received the information, but is not willing to propagate it,  ``Accepted'', the individual accepts the information and then propagates it, and ``Exhausted'', after propagating the information to their neighbours, the individual will lose his interests in it. The concurrent propagation of two types of information is modelled by the node states $S$, $I_1$ and $I_2$, where the effect of one information to the other can also be accounted for by the appropriate choice of the rate functions.	
	\item The concurrent spread of epidemic and information can be described by the node states susceptible and aware ($S_{\rm a}$), susceptible and not  aware ($S_{\rm na}$), infected and aware ($I_{\rm a}$), infected and not  aware ($I_{\rm na}$). This can also be extended with a treatment class ($T$).	
	\item Propagation of neuronal activity can be modelled by the node states  quiescent ($Q$) and active ($A$) with both excitatory and inhibitory neurons. The effect of different neurons to each other can be described by the rate functions.
\end{itemize}
See \cite{gleeson2013binary,porter2016dynamical,kiss2017book,ward2018general,ward2019exact} for more information about, and examples of, models within our framework. To translate such examples into our framework, it is necessary to associate the corresponding model rate constants with the set of functions that constitute the VSTM in \ref{eq:affinevstm}. Typical models have far fewer vertex-state transitions and rate constants than the general case, simplifying what needs to be considered.

Concerning the network structure one can specify a vertex partition, for which we list a few possibilities below.
\begin{itemize}
	\item If the nodes play a similar role in the network, then choosing a single partition ($P=1$) is a reasonable choice and corresponds to the `well mixed' case where network structure is essentially ignored, with the rates scaled by the network density. This case was derived in Section~\ref{subsec:single partition}.
	\item In some cases, nodes can be divided into two groups, for example highly and weakly connected nodes, then choosing two partitions, $P=2$, is natural.
	\item If the network is given by a bipartite graph, then the two node groups lead again to $P=2$.	
	\item The case of $k$-partite graphs can be handled with $P=k$ partitions. We note that in the case of complete $k$-partite graphs the lumping is exact.	
	\item A natural choice for vertex partitions is based on the node degrees, i.e. two nodes are in the same vertex partition if their degree is equal.
\end{itemize}

Once the node dynamics is specified through the rate functions \eqref{eq:affinevstm} and the vertex partitions are given according to Definitions \ref{def:vpppms} and \ref{def:vpal}, one can determine the generator \eqref{eq:qijgeneral} as follows.
The macro-states $\mathbf{s}^{[i]}$ are $M\times P$ matrices defined in Definition \ref{def:vpal} by specifying the number of vertices in a given state being in a given partition. The coefficients $\zeta_m^{\mathcal{A},\mathcal{B}}$ are determined by the transition functions given in \eqref{eq:affinevstm}.
Considering only the few examples of node-dynamics and partitions listed above, several dozens of models can be derived based on our main Theorem \ref{thm:genrates} since each node dynamics can be combined with each partition.

Note that the generator in \eqref{eq:qijgeneral} is an $n\times n$ matrix, where $n$ is given in \eqref{eq:lumpedn}. For example, when we have only one partition, $P=1$ and two node-states, $M=2$, then $n=N+1$ which is significantly smaller than the full system size $2^N$. In the case of $M=3$ node-states we have $n=\mathcal{O}(N^2)$ which can be large but still much smaller than the full system size $3^N$. In the case of $P=2$ partitions and $M=2$ node-states one has $n=\mathcal{O}(N^2)$, while for $M=3$ node-states it is $n=\mathcal{O}(N^4)$. Thus the size of the lumped system is polynomial in $N$ compared to exponential for the full system. We will get a further significant decrease in system size in the next section. Note however that using the generator \eqref{eq:qijgeneral} provides information about the probability distribution over macrostate-space, and in systems with absorbing states, one can estimate the probability of and time to absorption \cite{ward2019exact}.

\section{Large $N$ limit of density dependent population processes}
\label{sec:largeN}
The approximate lumping process derived in Section~\ref{sec:Ppartitions} significantly reduces the number of equations that need to be considered, while introducing some approximation error. However, the number of equations will often still be very large, see the calculations at the end of the previous section. 
However, our approximation \eqref{eq:qijgeneral} is a ``density-dependent population process'' \cite{ethier2009markov}, which in the large $N$ limit converges to a smaller system of $M\times P$ differential equations. We will now describe this in more detail.

\subsection{Mean-field limit of a general density dependent process}
For positive integer $N$, $\xi\in\mathbb{Z}^d$, $E\subset\mathbb{R}^d$, and a collection of non-negative functions $\lambda_{\xi}:E\rightarrow\mathbb{R}_{\ge0}$, let $$E_N=E\cap\{k/N\mid k\in\mathbb{Z}^d\},$$ and assume that $y\in E_N$ and $\lambda_{\xi}(y)>0$ imply that $y+\xi/N\in E_N$. Then a density dependent family corresponding to $\lambda_{\xi}$ is a sequence $\{Y_N\}$ of Markov jump processes such that $Y_N$ has state-space $E_N$ and transition intensities $$q^{(N)}_{x,y}=N\left[\lambda_{N(y-x)}(x)+\mathcal{O}\left(\frac{1}{N}\right)\right],\;x,y\in E_N.$$
Let $$F(y)=\sum_{\xi\in\mathbb{Z}^d} \xi\lambda_{\xi}(y),$$
then provided that for each compact $K\subset E$,
$$\sum_{\xi\in\mathbb{Z}^d}|\xi|\sup_{y\in K}\lambda_{\xi}(y)<\infty,$$
and there exists an $M_K>0$ such that
$$|F(x)-F(y)|\le M_K|x-y|,\quad x,y\in K,$$
then in the limit $N\rightarrow\infty$ there is almost sure convergence between the Markov chain jump process $Y_N(t)$ and $y(t)$, where $y(t)$ is the solution to system of differential equations $$\dot{y}=F(y).$$
A more precise statement can be found in Ethier and Kurtz's book \cite{ethier2009markov}.

As a simple example of a density dependent family, consider the SIS model with population $N$. Suppose that if there are $i$ susceptible individuals and $j=N-i$ infected individuals then the infection rate is $$q_{(i,j)(i-1,j+1)}=N\beta\frac{i}{N}\frac{j}{N},$$ and the recovery rate is $$q_{(i,j)(i+1,j-1)}=N\gamma\frac{j}{N}.$$ This corresponds to the standard SIS stochastic compartmental model birth-death process. The possible values of $\xi$ are $(-1,+1)$ and $(+1,-1)$, so for $y=(y_1,y_2)$ we have  $\lambda_{(-1,+1)}(y)=\beta y_1y_2$ and $\lambda_{(+1,-1)}(y)=\gamma y_2$. Consequently
\begin{align*}
F(y)&=\sum_{\xi}\xi\lambda_{\xi}(y)
=(-\beta y_1y_2+\gamma y_2,+\beta y_1y_2-\gamma y_2).
\end{align*}
Thus the familiar compartmental SIS model ODEs are the large $N$ limit of the corresponding stochastic birth-death process.

\subsection{Limiting equations of vertex-partition lumping}
To connect the approach presented in the previous subsection to vertex-partition lumping, we first introduce some notation.
Let $\mathbf{e}(\mathcal{W}_m,p)\in\{0,1\}^{M\times P}$ be a matrix whose $m,p$th entry is 1, and all other entries are zero.
Let $\xi^{\mathcal{A},\mathcal{B}}_q=\mathbf{e}(\mathcal{B},q)-\mathbf{e}(\mathcal{A},q)$, so a transition from $\mathbf{s}$ to $\mathbf{s}+\xi^{\mathcal{A},\mathcal{B}}_q$ corresponds to a vertex in partition $V_q$ changing from vertex-state $\mathcal{A}$ to $\mathcal{B}$. Then we can write the transition rate \eqref{eq:qijgeneral} as
\begin{align}
\mathbf{q}_{\mathbf{s},\mathbf{s}+\xi^{\mathcal{A},\mathcal{B}}_q}=N\frac{\mathbf{s}_{\mathcal{A},q}^{[i]}}{N}\left\{
\zeta_0^{\mathcal{A},\mathcal{B}}+\frac{N}{N_q}\sum_{p=1}^P
\left(\frac{\sum_{v\in V_q}d^v_p}{N_r-\delta_{q,p}}\right)
\left[\sum_{m=1}^{M}\zeta_m^{\mathcal{A},\mathcal{B}}
\frac{\left(\mathbf{s}_{m,p}^{[i]}-\delta_{\mathcal{A},\mathcal{W}_m}\delta_{q,p}\right)}{N}
\right]\right\}.
\label{eq:qsxiN}
\end{align}
We consider the density $\mathbf{y}=\mathbf{s}/N$
and we assume that
\begin{equation}
\lim_{N\rightarrow\infty}\frac{N_p}{N}=\theta_p>0
\quad\text{and}\quad
\lim_{N\rightarrow\infty}\sum_{v\in V_q}\frac{d_p^v}{N_p}=z_{p,q}\ge0,
\label{eq:NpN}
\end{equation}
with $\theta_p$ and $z_{p,q}$ finite. Note that $\theta_p$ is the fraction of vertices in partition $p$, so for this to be finite in the large $N$ limit, all vertex partitions must scale with $N$. Also, $z_{p,q}$ is the average number of edges between partitions $p$ and $q$, averaged over the number of vertices in vertex partition $p$. We also have
\begin{align*}
\sum_{v\in V_q}\frac{d_p^v}{N_q-1}&=\left(\sum_{v\in V_q}\frac{d_p^v}{N_q}\right)\frac{N_q}{N_q-1}
=\left(\sum_{v\in V_q}\frac{d_p^v}{N_q}\right)\left(1+\frac{1}{N\left[\frac{N_q}{N}-\frac{1}{N}\right]}\right).
\end{align*}
Thus we can write \eqref{eq:qsxiN} as
\begin{equation*}
\mathbf{q}_{\mathbf{s},\mathbf{s}+\xi^{\mathcal{A},\mathcal{B}}_q}=N\left[\lambda_{\xi^{\mathcal{A},\mathcal{B}}_q}\left(\frac{\mathbf{s}}{N}\right)+\mathcal{O}\left(\frac{1}{N}\right)\right],
\end{equation*}
where
\begin{equation}
\lambda_{\xi^{\mathcal{A},\mathcal{B}}_q}(\mathbf{y})=\mathbf{y}_{\mathcal{A},q}\left(\zeta_0^{\mathcal{A},\mathcal{B}}
+\sum_{p=1}^P\frac{z_{p,q}}{\theta_q}\sum_{m=1}^M
\zeta_m^{\mathcal{A},\mathcal{B}}\mathbf{y}_{m,p}\right).
\label{eq:lambda-general}
\end{equation}
Note that $\lambda_{\xi^{\mathcal{A},\mathcal{B}}_q}$ returns a scalar but $\xi^{\mathcal{A},\mathcal{B}}_q$ is a matrix and it identifies the $\mathcal{A},\mathcal{B}\in\mathcal{W}$ and $0\le q\le P$ used in \eqref{eq:lambda-general}. In the large $N$ limit, we have a system of matrix differential equations for $\mathbf{y}$, given by
\begin{equation}
\dot{\mathbf{y}}=\sum_{\mathcal{A}\in\mathcal{W}}\sum_{\mathcal{B}\ne\mathcal{A}}\sum_{q=1}^P
\xi^{\mathcal{A},\mathcal{B}}_q\lambda_{\xi^{\mathcal{A},\mathcal{B}}_q}(\mathbf{y}).
\label{eq:mean-field-general}
\end{equation}

\subsection{Application of the mean-field-limit approach}
Here we present the applicability of the mean-field limit derivation shown in the previous subsection. Using our approach one can derive the mean-field limit equation \eqref{eq:mean-field-general} for each node-dynamic and vertex partition listed in Section \ref{subsec:application_lumping}. The model dynamics specify the transitions $\xi^{\mathcal{A},\mathcal{B}}_q$ and the corresponding transition rates $\zeta_m^{\mathcal{A},\mathcal{B}}$. 
The vertex-set partition can also be chosen quite generally, subject to the conditions \eqref{eq:NpN}, which assume that as $N$ tends to infinity, the proportion of vertices in each partition, $\theta_p$, and the proportion of edges between partitions, $z_{p,q}$, tend to constant values. 
Once these constants have been determined, one can formulate transition functions $\lambda_{\xi^{\mathcal{A},\mathcal{B}}_q}$ for each transition $\xi^{\mathcal{A},\mathcal{B}}_q$ using \eqref{eq:lambda-general}. The dependent variable in the mean-field equations \eqref{eq:mean-field-general} is the time-dependent matrix $\mathbf{y}$, which is the scaled version of the lumped variables, i.e. $\mathbf{y}=\mathbf{s}/N$, meaning that $\mathbf{y}_{m,p}$ is the proportion of vertices in vertex-state $\mathcal{W}_m$ and in partition cell $V_p$. Thus the number of differential equations in \eqref{eq:mean-field-general} is $M\times P$. 
Considering only the few examples of node-dynamics and partitions listed in Section~\ref{subsec:application_lumping}, several dozens of mean-field equations can be derived based on our approach since each node dynamic can be combined with each partition.

\section{Degree-based mean-field}
\label{sec:dbMF}
We now relate the large $N$ limit of our vertex-partition lumping to the well known degree-based mean-field.
Consider the vertex partition where for each $k>0$, we have that $V_k$ is the set of vertices with degree $k$. Thus $N_k$ is the number of vertices with degree $k$ and $p_k=N_k/N$ is the fraction of vertices with degree $k$, i.e. the degree distribution. We assume that $p_k>0$ for all $k$ (although the same approach could be applied to cases where there are $k$ such that $p_k=0$) and that there is a maximum degree $k_{\rm max}$, so that the vertex-set is partitioned into a finite number of cells. We will call a partition based on vertex degree ``degree-based mean-field'' and we will derive the ODEs for the SIS model. Let $\mathcal{W}_1$ correspond to susceptible nodes and $\mathcal{W}_2$ correspond to infected nodes. Recall that for the SIS model we have
 $\zeta_{2}^{1,2}=\beta$ and $\zeta_0^{2,1}=\gamma$, and all other $\zeta_m^{\mathcal{A},\mathcal{B}}=0$. Since $\mathbf{y}_{1,k}+\mathbf{y}_{2,k}=p_k$, we will refer to the fraction of infected nodes as $y_k=\mathbf{y}_{2,k}$, from which we can infer $\mathbf{y}_{1,k}=p_k-y_k$. Thus we only need to write the differential equations for $y_k$, and we will write $y=(y_1,y_2,\dots,y_{k_{\rm max}})$. Consequently we only need the second column of the matrices $\xi^{\mathcal{A},\mathcal{B}}_q$. Hence let $\xi_k$ be a vector of zeros with a one in the $k$th entry, then $\xi_k$ corresponds to an infection and $-\xi_k$ corresponds to a recovery. Using \eqref{eq:lambda-general}, the large $N$ infection rate is,
 $$\lambda_{\xi_k}(y)=\beta(p_k-y_k)\sum_{k'}\frac{z_{k',k}}{p_k}y_{k'},$$
 the large $N$ recovery rate is
 $$\lambda_{-\xi_k}(y)=\gamma y_k,$$
and hence the evolution equations are
\begin{equation}
\dot{y}_k=-\gamma y_k+\beta(p_k-y_k)\sum_{k'}\frac{z_{k',k}}{p_k}y_{k'}.
\label{eq:dbmf}
\end{equation}

We will now show that \eqref{eq:dbmf} is equivalent to the Eames and Keeling \cite{eames2002modeling} and Pastor-Satorras and Vespignani \cite{pastor2001epidemic} degree-based mean-field approximations. Eames and Keeling use $[I^k]$ to denote the number of infected nodes with $k$ neighbours and $[S^kI^{k'}]$ to denote the number of partnerships between a susceptible node with $k$ partners and an infected node with $k'$ partners. The dynamics is then described by
$$\frac{d[I^k]}{dt}=-\gamma[I^k]+\beta\sum_{k'}[S^kI^{k'}],$$
and this is closed with the approximation
$$[S^kI^{k'}]\approx\frac{[S^k]}{N_k}\times\frac{[I^{k'}]}{N_{k'}}\times[kk'],$$
where $[kk']$ is the number of partnerships between individuals with $k$ and $k'$ partners. Thus the Eames and Keeling degree-based mean-field is
\begin{equation}
\frac{d[I^k]}{dt}=-\gamma[I^k]+\beta\sum_{k'}\frac{[S^k]}{N_k}\frac{[I^{k'}]}{N_{k'}}[kk'].
\label{eq:EKdbmf}
\end{equation}
Let $y_k=[I^k]/N$, and noting that since $[S^k]+[I^k]=N_k$ we have $[S^k]/N=p_k-y_k$, then dividing \eqref{eq:EKdbmf} through by $N$ yields
$$\dot{y}_k=-\gamma y_k +\beta(p_k-y_k)\sum_{k'}\frac{[kk']}{p_kN_{k'}}y_{k'}.$$
Since $[kk']$ is the number of edges between degree $k$ and degree $k'$ vertices, we have
$$[kk']=\sum_{v\in V_k}d^v_{k'},$$
and consequently, using \eqref{eq:NpN}, the Eames and Keeling degree-based mean-field is equivalent to \eqref{eq:dbmf}. It is also easy to show that the Eames and Keeling degree-based mean-field is equivalent to the Pastor-Satorras and Vespignani \cite{pastor2001epidemic} degree-based mean-field, and we provide details of this in the Supplementary Information.

\section{Individual-based mean-field}
\label{sec:IBMF}
We will now show how our approach can be used to derive individual-based mean-field approximations. Consider a graph with $N$ vertices and $N_{\rm e}$ isomorphic copies of this graph. We will call the graph that is copied the base graph, and thus the collection of all copies of the base graph consists of $NN_{\rm e}$ vertices. We will apply approximate lumping to this collection of graphs by choosing the partition of vertices, $\Pi_V$, so that each partition cell $V_i$ has the corresponding vertex in each of the $N_{\rm e}$ copies of the graph. Thus there are $P=N$ partition cells, and we can increase the total number of vertices by increasing $N_{\rm e}$, the number of copies of the base graph. Applying our approximate lumping approach produces an ensemble average over the collection of isomorphic base graphs.

To apply our approximate lumping approach, we need to compute the fraction of vertices in each vertex partition cell, $\theta_i$, and the mean number of edges between vertex partition cells $i$ and $j$, $z_{i,j}$, as defined in \eqref{eq:NpN}. Note that our notation is slightly modified here, where we are taking the limit $N_{\rm e}\rightarrow\infty$ (and hence the total number of vertices in the collection of graphs) rather than $N$. Since there are $N_{\rm e}$ vertices in each partition cell and $NN_{\rm e}$ vertices in total, we have that $\theta_i=1/N$. Let $\mathbf{A}$ be the adjacency matrix of the base graph, so the component $\mathbf{A}_{ij}$ is one if vertices $i$ and $j$ are connected in the base graph and zero otherwise. Thus if $V_i\in\Pi_{V}$, then the number of neighbours of vertex $v\in V_i$ that are in the partition cell $V_j$ is $d^v_j=\mathbf{A}_{ij}$ (i.e. one if $i$ and $j$ are connected in the base graph, and zero otherwise). We will assume for $v\in V_i$ that $d^v_i=0$, i.e. there are no self-loops in the base graph. Consequently, since $N_i=N_{\rm e}$  for all $V_i\in \Pi_V$, we have
\begin{align*}
z_{i,j}&=\lim_{N_{\rm e}\rightarrow\infty}\sum_{v\in V_i}\frac{d^v_j}{N_{\rm e}}
=\mathbf{A}_{ij}.
\end{align*}

We are now in a position to make use of \eqref{eq:lambda-general}, but note that following the approach described in Section~\ref{sec:largeN}, we would use the density variable $\mathbf{y}=\mathbf{s}/(NN_{\rm e})$, i.e. the total fraction of vertices in each vertex-state and vertex partition. However, individual-based mean-field approximations are typically based on the `probability' that a vertex is in a given vertex-state. We can obtain a similar quantity here by instead using $\mathbf{y}=\mathbf{s}/N_{\rm e}$, i.e. the fraction of base graphs in which each vertex is in each vertex-state. We again let $\xi$ correspond to a vertex in $V_i$ changing from $\mathcal{A}$ to $\mathcal{B}$ and so \eqref{eq:qsxiN} becomes
\begin{equation*}
\mathbf{q}_{\mathbf{s},\mathbf{s}+\xi^{\mathcal{A},\mathcal{B}}_i}=N_{\rm e}\frac{\mathbf{s}_{\mathcal{A},i}}{N_{\rm e}}\left(\zeta_0^{\mathcal{A},\mathcal{B}}+
\sum_{j\ne i}\mathbf{A}_{ij}
\left[\sum_{m=1}^M\zeta_m^{\mathcal{A},\mathcal{B}}\frac{\mathbf{s}_{m,j}}{N_{\rm e}}\right]+\mathcal{O}(1/N_{\rm e})\right).
\end{equation*}
Consequently, using $\mathbf{y}=\mathbf{s}/N_{\rm e}$, the large $N_{\rm e}$ transition rate is
$$\lambda_{\xi^{\mathcal{A},\mathcal{B}}_i}(\mathbf{y})=\mathbf{y}_{\mathcal{A},i}\left(\zeta_0^{\mathcal{A},\mathcal{B}}+
\sum_{j\ne i}\mathbf{A}_{ij}
\left[\sum_{m=1}^M\zeta_m^{\mathcal{A},\mathcal{B}}\mathbf{y}_{m,j}\right]\right).$$
From this we recognise the transition rates that appear in standard individual-based mean-field equations, in which pairwise interaction terms are included according to the graph's adjacency matrix. For example, for the SIS model the evolution equation for the `probability' $y_i$ that vertex $i$ is infected is \cite{van2011n,kiss2017book}
$$\dot{y}_i=-\gamma y_i+\beta (1-y_i) \sum_{j\ne i} \mathbf{A}_{ij} y_j.$$

\section{The configuration model}
\label{sec:configmodel}
In the previous section, we saw how our approximate lumping approach could be applied to isomorphic copies of a given graph to obtain individual-based mean-field approximations. In this section we consider how this approach can be extended to derive mean-field approximations for dynamics on families of graphs, specifically configuration models.

Suppose that the vertices $V$ of a graph are labelled $1,2,\dots,N$, then a degree sequence is a sequence of integers $d_1,d_2,\dots,d_N$, such that $\sum_id_i=2M$ and for each $i$, $d_i\le N-1$. A random graph can be constructed from a given degree sequence by allocating each node $i$ with $d_i$ `stubs' and then picking pairs of stubs at random without replacement from the collection of all unpaired stubs to form edges in the graph. Such a graph may have self-edges and multi-edges. While the theory we have developed assumed simple networks, we conjecture that extending this approach to graphs with self- and multi-edges only introduces $\mathcal{O}(1/N)$ corrections to \eqref{eq:qijgeneral} in the large $N$ limit. Here we consider the family of all such graphs for a given degree sequence, and we will refer to these as configurations (of the associated degree sequence). The number of configurations can be computed fairly easily by considering the combinations of ways that pairs of stubs can be chosen, while accounting for the fact that the order in which pairs of stubs are selected is not important, which yields
\begin{align*}
    N_{\rm CM}(M)&=\frac{1}{M!}\binom{2M}{2}\binom{2M-2}{2}\cdots\binom{2}{2}
    =\frac{2M!}{2^MM!}.
\end{align*}
Thus there are $NN_{\rm CM}(M)$ vertices in the collection of configurations.

It is also easy to show that for a pair of vertices $i,j$, the average number of edges between $i$ and $j$ across all configurations is
$$\frac{d_id_j}{2M-1}.$$
To see this, note that there are $d_id_j$ ways to match each of the $d_i$ stubs to each of the $d_j$ stubs. Having used a pair of stubs to do this, there are then $2M-2$ stubs that remain to be matched, and there are $N_{\rm CM}(M-1)$ ways to do this. Note that some of these remaining pairs of stubs may also lead to edges between $i$ and $j$, but this exactly accounts for all possible multi-edges between $i$ and $j$. Thus
$$d_id_j\frac{N_{\rm CM}(M-1)}{N_{\rm CM}(M)},$$
and cancellation leads to the result.

To apply the vertex-partition lumping to the configuration model, we need to consider the large $N$ limit. Thus we suppose that for a given $N$, the degree sequence is sampled at random and the degrees of the $N$ vertices are $d^N_1,d^N_2,\dots,d^N_N$, where there is an $M_N$ such that $\sum_id^N_i=2M_N$. Furthermore, we assume that for each $0<k\le k_{\rm max}$, the number of vertices of degree $k$ is $n_k$ and
$$\lim_{N\rightarrow\infty}\frac{n_k}{N}=p_k,$$
where $p_k>0$ for each $k$ and $\sum_kp_k=1$. 
Thus in the large $N$ limit, the mean degree of the network will converge to $z=\sum_kkp_k$.

We again consider a vertex-partition based on degree for the family of configurations of $N$ vertices. Thus for finite $k_{\rm max}$ the vertex partition is $\Pi_V=\{V_1,V_2,\dots,V_{k_{\rm max}}\}$. As described already, there are $NN_{\rm CM}(M_N)$ vertices across the family of configurations and hence there are $N_k=|V_k|=n_kN_{\rm CM}(M_N)$ vertices of degree $k$. In order to apply our general formula \eqref{eq:lambda-general}, it remains to compute $z_{k',k}$, the average number of edges between vertices of degree $k$ and vertices of degree $k'$, with respect to the number of vertices of degree $k'$. We have seen that the average number of edges between nodes with degree $k$ and $k'$ is $kk'/(2M_N-1)$, so the total number of edges between nodes of degree $k$ and $k'$ across all configurations is this multiplied by the total number of vertices of degree $k$ and the number of vertices of degree $k'$ in a single configuration. Thus it follows that 
\begin{align}
z_{k',k}&=\lim_{N\rightarrow\infty}\sum_{v\in V_k}\frac{d^v_{k'}}{N_{k'}}
=\lim_{N\rightarrow\infty}\frac{kk'}{2M_N-1}\frac{N_kn_{k'}}{N_{k'}}
=\frac{kk'p_k}{z}.\label{eq:zkpk}
\end{align}

Focusing on the SIS model and using the notation from Section~\ref{sec:dbMF}, we can substitute \eqref{eq:zkpk} into \eqref{eq:dbmf}, which yields
\begin{equation*}
\dot{y}_k=-\gamma y_k+\beta(p_k-y_k)\sum_{k'}\frac{kk'}{z}y_{k'}.
\end{equation*}
Furthermore, substituting $y_k=\rho_k p_k$, rearranging and rescaling recovers the uncorrelated degree-based mean-field \cite{pastor2001epidemic}.

\section{Discussion}
\label{sec:discussion}
In this paper we have derived mean-field approximations for a broad class of dynamical processes on networks directly from their exact Markov chain description. We have done this using the method of approximate lumping, where macrostates are defined in terms of the number of vertices in each vertex-state in subsets of vertices that form a partition of the vertex set. We have proved that this approach results in a density dependent population process, from which the large $N$ limiting behaviour can be described in terms of a relatively small number of equations, specifically $M\times P$ equations, where $M$ is the number of vertex states and $P$ is the size of the vertex partition. We have shown how this approach can be used to derive degree-based and individual-based mean-field approximations.

Given how involved the direct calculations are, it is surprising and impressive that we recover exactly the well-known degree-based and individual-based mean-field approximations. However, we emphasise that the use of approximate lumping means that not only is the averaging process clear, it tells us under what circumstances the approximation would be exact, i.e. the Markov chain that corresponds to the density dependent population process. Our approach also highlights that there are two sources of error. The main, and uncontrolled, source of error arises from the choice of vertex-partition that defines the partition of microstate-space in the approximate lumping. The second arises from the large $N$ limit, but this vanishes as $N$ becomes large.

Our methodology formalises the process of obtaining mean-field equations for a given dynamical model, eliminating the need to rely on intuitive probabilistic arguments. This approach also explicitly reveals the type of averaging represented by mean-field approximations on networks.
Our results apply in a broad range of cases, allowing researchers working with new models---within the class of single-vertex transition models with local, affine transition functions considered here---to easily obtain a range of mean-field approximations, in a controlled manner.

The mean-field approximations that we have derived using our approximate lumping approach ignore dynamical correlations between the states of neighbouring vertices. Such correlations are taken into account in higher-order, edge-based mean-field approximations \cite{kiss2017book}, that are typically derived using moment closure arguments.
In a follow-up paper we will show how our approach can also be adapted to derive such edge-based mean-field approximations, extending its generality.
A high-accuracy form of mean-field approximation are ``approximate master equations'' \cite{gleeson2011high,gleeson2013binary}, which are based on the number of susceptible/infected nodes of degree $k$ with $m$ infected neighbours. While it may seem that these might be derived from our approach, we have been unable to do so because it is not clear how the counts change when vertices recover or become infected. In particular, it seems that one must know the degree and number of infected neighbours of the neighbours of a vertex that changes vertex-state.

Our ambition is to quantify mean-field approximation error in terms of the dynamics and network structure. Our approach in this paper highlights that the main source of error results from the choice of vertex-partition and how far the corresponding lumping is from being exact. Establishing quantitative estimates for the resulting error is a critical open challenge. An attempt to do this is described in \cite{ward2022micro} for the case of the SISa model and where vertices are not partitioned, but in this case the error could not be entirely unravelled from the full Markov chain. However, it may be that hierarchies of approximations can be constructed in which the error decays monotonically \cite{khudabukhsh2019approximate}. In this paper we have shown how our approach can be applied to families of graphs, and this suggests that it may be also possible to use our approach to derive mean-field approximations for models in which the network co-evolves with the vertex-state dynamic. We would also like to generalise our approach to models with nonlinear VSTMs. 


\input{paper.bbl}
\clearpage


\section*{Supplementary material}

**SM6. Sum-product property.**

THEOREM SM6.1. *The* sum-product property *is that*

$$\prod_{p=1}^{P}\sum_{i_p=1}^{r_p} a_{i_p}^p = \sum_{i_1=1}^{r_1}\sum_{i_2=1}^{r_2}\cdots\sum_{i_P=1}^{r_P}\prod_{p=1}^{P} a_{i_p}^p.$$

*Proof.* We prove this by induction. This is trivially true for $P=1$. We assume it is true for $P$, then for $P+1$ we have

$$\prod_{p=1}^{P+1}\sum_{i_p=1}^{r_p} a_{i_p}^p = \left(\prod_{p=1}^{P}\sum_{i_p=1}^{r_p} a_{i_p}^p\right)\left(\sum_{i_{P+1}=1}^{r_{P+1}} a_{i_{P+1}}^{P+1}\right),$$

$$= \left(\sum_{i_1=1}^{r_1}\sum_{i_2=1}^{r_2}\cdots\sum_{i_P=1}^{r_P}\prod_{p=1}^{P} a_{i_p}^p\right)\left(\sum_{i_{P+1}=1}^{r_{P+1}} a_{i_{P+1}}^{P+1}\right),$$

$$= \sum_{i_{P+1}=1}^{r_{P+1}} a_{i_{P+1}}^{P+1}\left(\sum_{i_1=1}^{r_1}\sum_{i_2=1}^{r_2}\cdots\sum_{i_P=1}^{r_P}\prod_{p=1}^{P} a_{i_p}^p\right),$$

$$= \sum_{i_1=1}^{r_1}\sum_{i_2=1}^{r_2}\cdots\sum_{i_{P+1}=1}^{r_{P+1}}\prod_{p=1}^{P+1} a_{i_p}^p. \qquad \square$$

As a consequence of the sum-product property, for lumped state $\mathbf{s}$ and neighbourhood count $\mathbf{n}^v$ of a vertex $v \in V$, we have

$$\prod_{p=1}^{P}\left[\sum_{\mathbf{n}_p^v|d_p^v} A(\mathbf{s}_p, \mathbf{n}_p^v)\right] = \sum_{\mathbf{n}_1^v|d_1^v}\sum_{\mathbf{n}_2^v|d_2^v}\cdots\sum_{\mathbf{n}_P^v|d_P^v}\prod_{p=1}^{P} A(\mathbf{s}_p, \mathbf{n}_p^v)$$

(SM6.1)
$$= \sum_{\mathbf{n}^v|d^v}\prod_{p=1}^{P} A(\mathbf{s}_p, \mathbf{n}_p^v).$$

**SM7. Generalised Vandermonde identity.** The Vandermonde identity can be generalised to multinomials. To see this, consider the multinomial theorem,

$$(x_1 + x_2 + \cdots + x_M)^N = \sum_{s_1+s_2+\cdots+s_M=N}\binom{N}{s_1, s_2, \ldots, s_M}\prod_{m=1}^{M} x_m^{s_m},$$

where $s_1, s_2, \ldots, s_M \geq 0$. Thus for $d < N$ we have

$$(x_1 + x_2 + \cdots + x_M)^{N-d}(x_1 + x_2 + \cdots + x_M)^d$$

$$= \left[\sum_{l_1+l_2+\cdots+l_M=N-d}\binom{N-d}{l_1, l_2, \ldots, l_M}\prod_{m=1}^{M} x_m^{l_m}\right]$$

$$\times \left[\sum_{n_1+n_2+\cdots+n_M=d}\binom{d}{n_1, n_2, \ldots, n_M}\prod_{m=1}^{M} x_m^{n_m}\right],$$

$$= \sum_{l_1+l_2+\cdots+l_M=N-d}\sum_{n_1+n_2+\cdots+n_M=d}\binom{d}{n_1, n_2, \ldots, n_M}\binom{N-d}{l_1, l_2, \ldots, l_M}\prod_{m=1}^{M} x_m^{l_m+n_m},$$

where $l_m, n_m \geq 0$. Thus instead of summing over $l_1 + l_2 + \cdots + l_M = N - d$ we can sum over $s_1 + s_2 + \cdots + s_M = N$ and use $l_m = s_m - n_m$ with the convention that the multinomial coefficients are zero if $s_m - n_m < 0$ or $\sum_m (s_m - n_m) \neq N - d$. Thus we have

$$\sum_{s_1+s_2+\cdots+s_M=N} \binom{N}{s_1, s_2, \ldots, s_M} \prod_{m=1}^M x_m^{s_m} = $$
$$\sum_{s_1+s_2+\cdots+s_M=N} \left[ \sum_{n_1+n_2+\cdots+n_M=d} \binom{d}{n_1, n_2, \ldots, n_M} \binom{N-d}{s_1-n_1, s_2-n_2, \ldots, s_M-n_M} \right]$$
$$\times \prod_{m=1}^M x_m^{s_m},$$

and the equality of the polynomials on the two sides implies that the coefficients of corresponding terms are identical, hence

$$\binom{N}{s_1, s_2, \ldots, s_M} = \sum_{n_1+n_2+\cdots+n_M=d} \binom{d}{n_1, n_2, \ldots, n_M} \binom{N-d}{s_1-n_1, s_2-n_2, \ldots, s_M-n_M}.$$

LEMMA SM7.1. *Let* $\mathbf{s}$ *be a lumped state and* $\mathbf{n}^v$ *a neighbourhood count of a vertex* $v \in V$, *then*

(SM7.1) $$\binom{N}{\mathbf{s}} = \sum_{\mathbf{n}^v | d^v} \prod_{p=1}^P A(\mathbf{s}_p, \mathbf{n}_p^v).$$

*Proof.* From the Vandermonde property of multinomials and (SM5.6), we have

$$\binom{N_p}{\mathbf{s}_p} = \sum_{\mathbf{n}_p^v | d_p^v} A(\mathbf{s}_p, \mathbf{n}_p^v).$$

Thus it follows that

$$\binom{N}{\mathbf{s}} = \prod_{p=1}^P \binom{N_p}{\mathbf{s}_p}$$
$$= \prod_{p=1}^P \left( \sum_{\mathbf{n}_p^v | d_p^v} A(\mathbf{s}_p, \mathbf{n}_p^v) \right).$$

The result then follows from application of the sum product property. $\square$

**SM8. Degree-based mean-field.** In this section we show that the Eames and Keeling degree-based mean-field is equivalent to the Pastor-Satorras and Vespignani [SM5] degree-based mean-field. The Eames and Keeling degree-based mean-field is

(SM8.1) $$\frac{d[I^k]}{dt} = -\gamma[I^k] + \beta \sum_{k'} \frac{[S^k]}{N_k} \frac{[I^{k'}]}{N_{k'}} [kk'],$$

where $[I^k]$ ($[S^k]$) denotes the number of infected (susceptible) nodes with $k$ neighbours, $N_k$ is the number of nodes with degree $k$, $[kk']$ is the number of partnerships between

individuals with $k$ and $k'$ partners, $\gamma$ is the recovery rate and $\beta$ is the infection rate. Pastor-Satorras and Vespignani use $\rho_k$ to denote the "average density" of infected nodes of degree $k$. The average fraction of infected nodes is then

$$\rho = \sum_k p_k \rho_k,$$

which suggests that $\rho_k = [I^k]/N_k$. The evolution of $\rho_k$ is described by

(SM8.2) $$\dot{\rho}_k = -\rho_k + \lambda k \left[1 - \rho_k\right] \frac{\sum_{k'} k' p_{k'} \rho_{k'}}{z},$$

where $\lambda = \beta/\gamma$ is the effective spreading rate, and $z = \sum_k k p_k$ is the mean degree. Boguñá et al [SM1] extend this to account for degree correlations, where they define $P(k'|k)$ to be the conditional probability that a node of degree $k$ is connected to a node of degree $k'$. The evolution of $\rho_k$ is then given by

(SM8.3) $$\dot{\rho}_k = -\rho_k + \lambda k \left[1 - \rho_k\right] \sum_{k'} P(k'|k) \rho_{k'}.$$

Now starting from the Ealing and Keemes degree-based mean field, we set $\rho_k = [I^k]/N_k$, so with $N_k = [S^k] + [I^k]$ we have $(1 - \rho_k) = [S^k]/N_k$, then dividing (SM8.1) through by $N_k$ yields

$$\dot{\rho}_k = -\gamma \rho_k + \beta \left[1 - \rho_k\right] \sum_{k'} \frac{[kk']}{N_k} \rho_{k'}.$$

Since $kN_k$ is the total number of edges connected to nodes of degree $k$, we define

$$P(k'|k) = \frac{[kk']}{kN_k},$$

and hence

$$\dot{\rho}_k = -\gamma \rho_k + \beta k \left[1 - \rho_k\right] \sum_{k'} P(k'|k) \rho_{k'},$$

which after time-rescaling yields (SM8.3). In the absence of degree correlations we have

$$P(k'|k) = \frac{k' N_{k'}}{zN} = \frac{k' p_{k'}}{z},$$

and hence we obtain (SM8.2).



\end{document}

%% file: VP_preamble_shared.tex
\usepackage{amsfonts,amssymb}
\usepackage{graphicx}

\ifpdf
  \DeclareGraphicsExtensions{.eps,.pdf,.png,.jpg}
\else
  \DeclareGraphicsExtensions{.eps}
\fi

\newsiamremark{remark}{Remark}
\newsiamremark{hypothesis}{Hypothesis}
\crefname{hypothesis}{Hypothesis}{Hypotheses}
\newsiamthm{claim}{Claim}

\headers{Mean-Field Approximation of Dynamics on Networks}{J.~A.~Ward, G. Tim\'{a}r, P. Simon}

\title{Mean-Field Approximation of Dynamics on Networks\thanks{Submitted to the editors DATE.
\funding{J.A. Ward and G. Tim\'{a}r acknowledge funding from the Leverhulme Trust Project Grant number RPG-2023-187. P.L. Simon acknowledges support from the Hungarian Scientific Research Fund, OTKA (grant no. 135241) and from ERC Synergy Grant No. 810115 - DYNASNET.}}}

\author{Jonathan A. Ward\thanks{University of Leeds, UK
  (\email{j.a.ward@leeds.ac.uk}).}
\and G\'{a}bor Tim\'{a}r\footnotemark[2]
\and P\'{e}ter L. Simon\thanks{Eotvos Lorand University, Hungary}}

\usepackage{amsopn}
